\documentclass{amsart}

\usepackage{amssymb}
\usepackage{graphicx}
\usepackage{enumerate}
\usepackage{amscd}
\usepackage{ascmac}
\usepackage[all]{xy}

\title{Modification rule of monodromies in $R_2$-move}
\author{Kenta Hayano}
\address{Department of Mathematics, Graduate School of Science, Osaka University, Toyonaka, Osaka 560-0043, Japan}
\email{k-hayano@cr.math.sci.osaka-u.ac.jp}

\theoremstyle{plain}
\newtheorem{thm}{Theorem}[section]

\newtheorem{lem}[thm]{Lemma}

\theoremstyle{definition}
\newtheorem{defn}[thm]{Definition}
\newtheorem{exmp}[thm]{Example}

\newtheorem{rem}[thm]{Remark}



\def\Ker{\operatorname{Ker}}

\def\Int{\operatorname{Int}}
\def\Diff{\operatorname{Diff}}
\def\res{\operatorname{res}}

\def\MCG{\operatorname{MCG}}

\begin{document}

\maketitle

\begin{abstract}

An $R_2$-move is a homotopy of wrinkled fibrations which deforms images of indefinite fold singularities like Reidemeister move of type II. 
Variants of this move are contained in several important deformations of wrinkled fibrations, flip and slip for example. 
In this paper, we first investigate how monodromies are changed by this move. 
For a given fibration and its vanishing cycles, we then give an algorithm to obtain vanishing cycles in one reference fiber of a fibration, which is obtained by applying flip and slip to the original fibration, in terms of mapping class groups. 
As an application of this algorithm, we give several examples of diagrams which were introduced by Williams \cite{Wil2} to describe smooth $4$-manifolds by simple closed curves of closed surfaces.

\end{abstract}

\section{Introduction}

Over the last few years, several new fibrations on $4$-manifolds were introduced and studied by means of various tools: singularity theory, mapping class groups, and so on. 
These studies started from the work of Auroux, Donaldson and Katzarkov \cite{ADK} in which they generalized the results of Donaldson \cite{Donaldson} and Gompf \cite{Gompf} on relation between symplectic manifolds and Lefschetz fibrations to these on relation between near-symplectic $4$-manifolds and corresponding fibrations, called {\it broken Lefschetz fibrations}. 
After their study, Perutz \cite{Perutz1}, \cite{Perutz2} defined the Lagrangian matching invariant for near-symplectic $4$-manifolds as a generalization of standard surface count of Donaldson and Smith \cite{Donaldson_Smith} for symplectic $4$-manifolds by using broken Lefschetz fibrations. 
Although this invariant is a strong candidate for geometric interpretation of the Seiberg-Witten invariant, even smooth invariance of this invariant is not verified so far. 
To prove this, we need to understand deformation (in the space of more general fibrations) between two broken Lefschetz fibrations. 
There are several results on this matter (see \cite{Lek}, \cite{Wil}, \cite{Gay_Kirby}, \cite{Gay_Kirby2011} and \cite{Wil2}, for example). 

On the other hand, broken Lefschetz fibrations themselves have been studied in terms of mapping class groups by looking at vanishing cycles. 
For example, classification problem of fibrations with particular properties were solved by means of this combinatorial method (see \cite{BK}, \cite{H} and \cite{H2}). 
It is known that every closed oriented $4$-manifold admits broken Lefschetz fibration (this kind of result first appeared in \cite{Gay_Kirby2007}, and then improved in \cite{AK}, \cite{Ba2} and \cite{Lek}). 
It is therefore natural to expect that broken Lefschetz fibrations enable us to deal with broader range of $4$-manifolds in combinatorial way, as we dealt with symplectic $4$-manifolds using Lefschetz fibrations. 
For the purpose of developing topology of smooth $4$-manifolds by means of mapping class groups, it is necessary to understand relation between several deformations appeared in study in the previous paragraph and vanishing cycles of fibrations. 

In this paper, we will pay our attention to a specific deformation of fibrations, called an {\it $R_2$-move}. 
In this move, the image of indefinite fold singularities are changed like Reidemeister move of type II (we will define this move in Section \ref{sec_BLFoverannulus}. See Figure \ref{changesingularloci}). 
In particular, the region with the highest genus fibers was cut off in this deformation. 
Furthermore, monodromies in this region might be changed by this move. 
This move appear in a lot of important deformations of fibrations. 
For example, {\it flip and slip}, which was first introduced by Baykur \cite{Ba2}, is application of flip twice followed by a variant of $R_2$-move. 
Another variant of $R_2$-move played a key role in the work of Williams \cite{Wil2}, which gave a purely combinatorial description of $4$-manifolds (which we will mention in Section \ref{sec_exampleWilliamsdiagram}). 

The main purpose of this paper is to understand how monodromies are changed by $R_2$-move. 
We will prove that modifications of monodromies in $R_2$-move can be controlled by an intersection of kernels of some homomorphisms (see Theorem \ref{keythm_monodromyalonggamma}). 
We will also give an algorithm to obtain vanishing cycles in a reference fiber of a fibration obtained by flip and slip in terms of mapping class group (see Theorem \ref{mainalgorithm}, \ref{mainalgorithmwithcusps}, \ref{mainalgorithmwithsection}, \ref{mainalgorithmwithcuspswithsection}, \ref{mainalgorithm_SPWF}, and \ref{mainalgorithm_SPWFwithsection}). 
Note that it is {\it not} easy to determine vanishing cycles in {\it one} reference fiber of the fibration obtained by applying flip and slip. 
Indeed, in this modification, two regions with the highest genus fibers are connected by a variant of $R_2$-move. 
It is easy to obtain vanishing cycles in fibers in the respective components since flip is a local deformation. 
However, we need to deal with a certain monodromy derived from a variant of $R_2$-move to understand how these fibers are identified (see also Remark \ref{rem_nontriviality}).  

In Section \ref{sec_preliminaries}, we will give several definitions and notations which we will use in this paper.
Sections \ref{sec_BLFoverannulus}, \ref{sec_mainalgorithm} and \ref{sec_algorithm_smallgenera} are the main parts of this paper. 
In Section \ref{sec_BLFoverannulus}, we will examine how monodromies are changed in $R_2$-moves. 
The results obtained in this section will play a key role in the following sections. 
In Sections \ref{sec_mainalgorithm} and \ref{sec_algorithm_smallgenera}, we will give an algorithm to obtain vanishing cycles of a fibration modified by flip and slip. 
We will first deal with fibrations with large fiber genera in Section \ref{sec_mainalgorithm}, and then turn our attention to fibrations with small fiber genera in Section \ref{sec_algorithm_smallgenera}. 
In Section \ref{sec_exampleWilliamsdiagram}, we will give a modification rule of a diagram Williams introduced, which will be called a {\it Williams diagram} in this paper, when the corresponding fibration is changed by flip and slip. 
We will then construct Williams diagrams of some fundamental $4$-manifolds, $S^4$, $S^1\times S^3$, $\mathbb{CP}^2\#\overline{\mathbb{CP}^2}$, and so on. 
Note that, as far as the author knows, these are the first non-trivial examples of Williams diagrams.

\vspace{.8em}

\noindent
{\bf Acknowledgments. }
The author would like to express his gratitude to Jonathan Williams for helpful discussions on Williams diagrams. 
The author is supported by Yoshida Scholarship 'Master 21' and he is grateful to Yoshida Scholarship Foundation for their support. 

\section{Preliminaries}\label{sec_preliminaries}

\subsection{Wrinkled fibrations}

We first define several singularities to which we will pay attention in this paper. 

\begin{defn}

Let $M$ and $B$ be smooth manifolds of dimension $4$ and $2$, respectively. 
For a smooth map $f:M\rightarrow B$, we denote by $\mathcal{S}_f\subset M$ the set of singularities of $f$. 

\begin{enumerate}

\item $p\in \mathcal{S}_f$ is called an {\it indefinite fold singularity} of $f$ if there exists a real coordinate $(t,x,y,z)$ (resp. $(s,w)$) around $p$ (resp. $f(p)$) such that $f$ is locally written by this coordinate as follows: 
\[
f: (t,x,y,z)\mapsto (s,w)=(t, x^2+y^2-z^2). 
\]

\item $p\in \mathcal{S}_f$ is called an {\it indefinite cusp singularity} of $f$ if there exists a real coordinate $(t,x,y,z)$ (resp. $(s,w)$) around $p$ (resp. $f(p)$) such that $f$ is locally written by this coordinate as follows: 
\[
f: (t,x,y,z)\mapsto (s,w)=(t, x^3-3tx+y^2-z^2). 
\]

\item We further assume that the manifolds $M$ and $B$ are oriented. 
$p\in \mathcal{S}_f$ is called a {\it Lefschetz singularity} of $f$ if there exists a complex coordinate $(z,w)$ (resp. $\xi$) around $p$ (resp. $f(p)$) compatible with orientation of the manifold $M$ (resp. $B$) such that $f$ is locally written by this coordinate as follows: 
\[
f: (z,w)\mapsto \xi = zw. 
\]

\end{enumerate}

\end{defn}

We can also define a definite fold singularities and definite cusp singularities. 
However, these singularities will not appear in this paper. 
We call an indefinite fold (resp. cusp) singularity a {\it fold} (resp. {\it cusp}) for simplicity. 

\begin{defn}

Let $M$ and $B$ be oriented, compact, smooth manifolds of dimension $4$ and $2$, respectively. 
A smooth map $f:M\rightarrow B$ is called a {\it wrinkled fibration} if it satisfies the following conditions: 

\begin{enumerate}

\item $f^{-1}(\partial B)=\partial M$, 

\item the set of singularities $\mathcal{S}_f$ consists of folds, cusps, and Lefschetz singularities, 

\end{enumerate}

A wrinkled fibration $f$ is called a {\it purely wrinkled fibration} if $f$ has no Lefschetz singularities. 

\end{defn}

\subsection{Mapping class groups and a homomorphism $\Phi_c$}

Let $\Sigma_g$ be a closed, oriented, connected surface of genus-$g$. 
We take subsets $A_i, B_j\subset \Sigma_g$. 
We define a group $\MCG{(\Sigma_g, A_1,\ldots A_n)}(B_1,\ldots, B_m)$ as follows: 
\[
\MCG{(\Sigma_g, A_1,\ldots, A_n)}(B_1,\ldots, B_m) = \left\{[T]\in \pi_0(\Diff^+{(\Sigma_g,A_1,\ldots, A_n)}, \text{id}) \hspace{.3em} | \hspace{.3em} T(B_j)=B_j\text{ for all $j$}\right\}, 
\]
where $\Diff^+{(\Sigma_g,A_1,\ldots, A_n)}$ is defined as follows: 
\[
\Diff^+{(\Sigma_g,A_1,\ldots, A_n)}= \left\{T:\Sigma_g\rightarrow \Sigma_g\text{: diffeomorphism} \hspace{.3em} | \hspace{.3em} T|_{A_i} = \text{id}_{A_i}\text{ for all $i$}\right\}. 
\]
In this paper, we define a group structure on the above group by multiplication {\it reverse to the composition}, that is, for elements $T_1,T_2\in\Diff^+{(\Sigma_g,A_1,\ldots, A_n)}$, we define the product $T_1\cdot T_2$ as follows: 
\[
T_1\cdot T_2 = T_2\circ T_1. 
\]
We define a group structure of $\MCG{(\Sigma_g, A_1,\ldots, A_n)}(B_1,\ldots, B_m)$ in the same way. 
For simplicity, we denote by $\mathcal{M}_g$ the group $\MCG{(\Sigma_g)}$.

Let $c\subset \Sigma_g$ be a simple closed curve. 
For a given element $\psi\in \MCG{(\Sigma_g)}(c)$, we take a representative $T:\Sigma_g\rightarrow \Sigma_g\in \Diff^+{(\Sigma_g)}$ preserving the curve $c$ setwise.  
The restriction $T|_{\Sigma_g\setminus c}: \Sigma_g\setminus c\rightarrow \Sigma_g\setminus c$ is also a diffeomorphism. 
Let $S_c$ be the surface obtained by attaching two disks with marked points at the origin to $\Sigma_g\setminus c$ along $c$. 
$S_c$ is diffeomorphic to $\Sigma_{g-1}$ with two marked points if $c$ is non-separating, or $S_c$ is a disjoint union of $\Sigma_{g_1}$ with a marked point and $\Sigma_{g_2}$ with a marked point for some $g_1,g_2$ if $c$ is separating. 
The diffeomorphism $T|_{\Sigma_g\setminus c}$ can be extended to a diffeomorphism $\tilde{T}: S_c\rightarrow S_c$. 
We define an element $\Phi_c^{\ast}([T])$ as an isotopy class of $\tilde{T}$, which is contained in the group $\MCG{(S_c, \{v_1,v_2\})}$, where $v_1,v_2$ are the marked points. 
By Proposition 3.20 in \cite{Farb_Margalit}, the following map is a well-defined homomorphism: 
\[
 \Phi_c^{\ast}: \MCG{(\Sigma_{g},c)}\rightarrow \MCG{(S_c,\{v_1,v_2\})}.  
\]
Furthermore, we define a homomorphism $\Phi_c$ on $\MCG{(\Sigma_g)}(c)$ as the composition $F_{v_1,v_2}\circ \Phi_c^{\ast}$, where $F_{v_1,v_2}: \MCG{(S_c, \{v_1,v_2\})}\rightarrow \MCG{(S_c)}$ is the forgetful map. 
The range of this map is $\mathcal{M}_{g-1}$ if $c$ is non-separating, $\mathcal{M}_{g_1}\times \mathcal{M}_{g_2}$ if $c$ is separating and $g_1\neq g_2$, and $(\mathcal{M}_{g_1}\times \mathcal{M}_{g_2})\rtimes \mathbb{Z}/2\mathbb{Z}$ if $c$ is separating and $g_1= g_2$. 
Note that Baykur has already mentioned relation between such homomorphisms and monodromy representations of simplified broken Lefschetz fibrations in \cite{Ba}.

\subsection{Several homotopies of fibrations}

In this subsection, we will give a quick review of some deformations of smooth maps from $4$-manifolds to surfaces which we will use in this paper.  
For details about this, see \cite{Lek} or \cite{Wil}, for example.




\subsubsection{Sink and Unsink}

Lekili \cite{Lek} introduced a homotopy which changes a Lefschetz singularity with indefinite folds into a cusp as in Figure \ref{baselocus_sink}. 
This modification is called a {\it sink} and the inverse move is called an {\it unsink}. 
We can always change cusps into Lefschetz singularities by unsink. 
However, we can apply sink only when $c_3$ corresponds to the curve $t_{c_1}(c_2)$, where $c_i$ is a vanishing cycle determined by $\gamma_i$, which is a reference path in the base space described in Figure \ref{baselocus_sink}. 

\begin{figure}[htbp]
\begin{center}
\includegraphics[width=75mm]{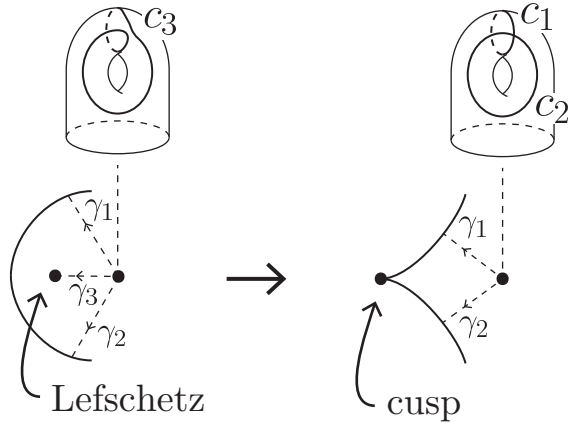}
\end{center}
\caption{Left: fibration with indefinite folds and a Lefschetz singularity. 
Right: fibration with a cusp. }
\label{baselocus_sink}
\end{figure}

\subsubsection{Flip, and "flip and slip"}

A homotopy called {\it flip} is locally written as follows: 
\[
f_s: \mathbb{R}^4\ni (t,x,y,z)\mapsto (t, x^4-x^2s+xt+y^2-z^2) \in \mathbb{R}^2. 
\]
The set of singularities $\mathcal{S}_{f_s}\subset \mathbb{R}^4$ corresponds to $\{(t,x,0,0)\in\mathbb{R}^4\hspace{.3em} | \hspace{.3em} 4x^3-2sx+t=0\}$. 
For $s<0$, this set consists of indefinite folds. 
For $s>0$, this set contains two cusps as in the right side of Figure \ref{baselocus_flip}. 

\begin{figure}[htbp]
\begin{center}
\includegraphics[width=100mm]{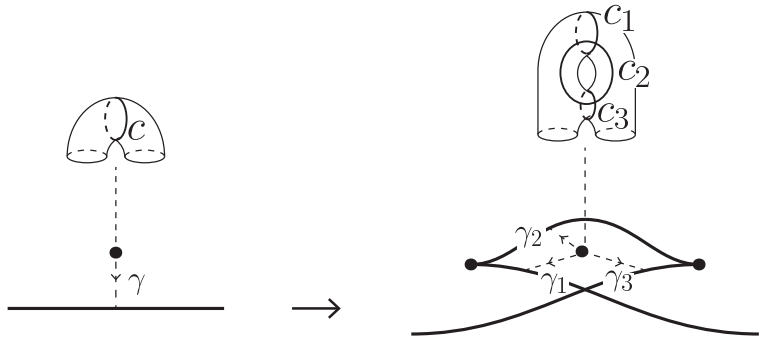}
\end{center}
\caption{Left: the image of singularities for $s<0$. Right: the image of singularities for $s>0$. 
$c_i$ describes a vanishing cycle determined by the reference path $\gamma_i$. 
As is described, $c_1$ is disjoint from $c_3$. }
\label{baselocus_flip}
\end{figure}

Baykur introduced in \cite{Ba2} a certain global homotopy, which is called a {\it flip and slip} in \cite{Ba2}, to make fibers of fibrations connected. 
This modification changes indefinite folds with circular image into circular singularities with four cusps (see Figure \ref{baselocus_flipandslip}). 
If a lower-genus regular fiber of the original fibration (i.e. a regular fiber on the inside of the singular circle of far left of Figure \ref{baselocus_flipandslip}) is disconnected, then this fiber becomes connected after the modification. 
If a lower-genus regular fiber is connected, this fiber becomes a higher genus fiber and the genus is increased by $2$.  

\begin{figure}[htbp]
\begin{center}
\includegraphics[width=120mm]{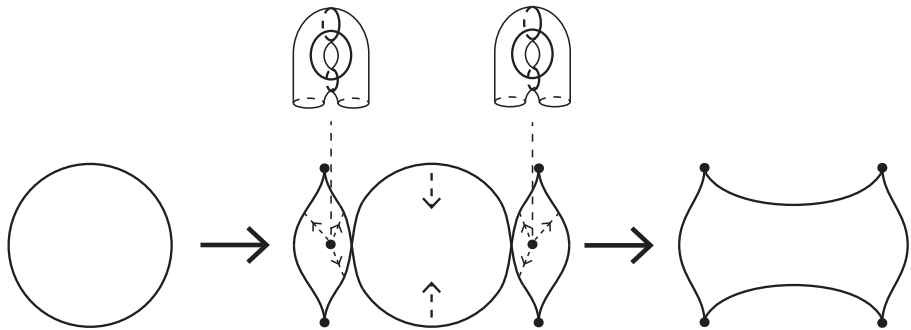}
\end{center}
\caption{The circle in far left figure describes the image of singularities of the original fibration. 
After applying flip twice, we change the fibration by a certain homotopy which makes the singular image a circle in the base space. }
\label{baselocus_flipandslip}
\end{figure}

\begin{rem}\label{rem_nontriviality}

It is {\it not} easy to obtain vanishing cycles of the fibration in {\it one} reference fiber obtained by applying flip and slip. 
Indeed, to find the vanishing cycles, we need to know how to identify two regular fibers on the regions with the highest genus fiber in the center of Figure \ref{baselocus_flipandslip}. 
As we will show in the following sections, this identification depends on the choice of homotopies, especially the choice of "slip" (from the center figure to the right figure in Figure \ref{baselocus_flipandslip}). 

\end{rem}

We remark that such a modification can be also applied when the set of singularities of the original fibration contains cusps. 
We first apply flip twice between two consecutive cusps. 
We then apply slip in the same way as in the case that the original fibration contains no cusps (see Figure \ref{flipandslipwithcusps}). 
We also call this deformation {\it flip and slip}. 

\begin{figure}[htbp]
\begin{center}
\includegraphics[width=125mm]{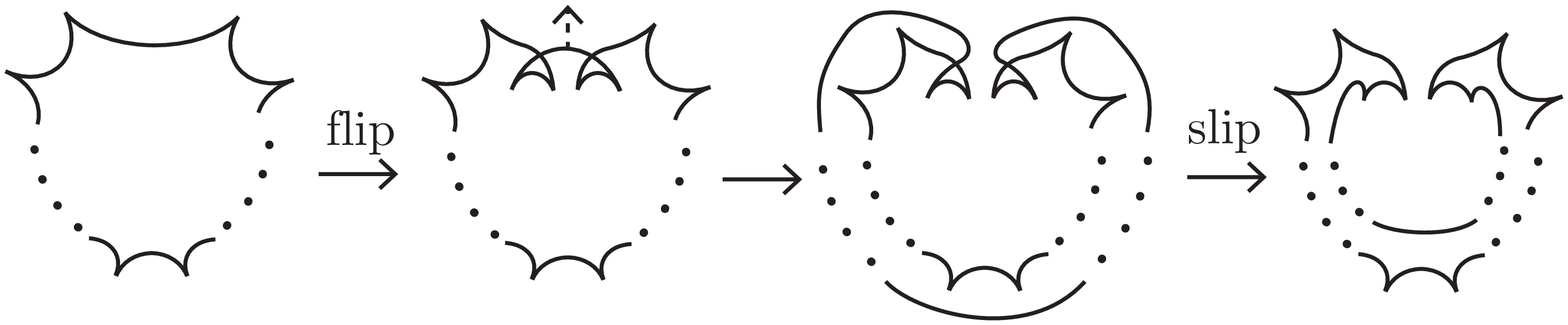}
\end{center}
\caption{base loci in flip and slip when the original fibration has cusps. }
\label{flipandslipwithcusps}
\end{figure}

\section{A fibration over the annulus with two components of indefinite folds}\label{sec_BLFoverannulus}

Let $N$ be a $3$-manifold obtained by $1$-handle attachment to $\Sigma_g\times I$ followed by $2$-handle attachment whose attaching circle is non-separating and is disjoint from the belt circle of the $1$-handle. 
$N$ has a Morse function $h:N\rightarrow I$ with two singularities; one is the center of the $1$-handle $p_1\in N$ whose index is $1$, and the other is the center of the $2$-handle $p_2\in N$ whose index is $2$. 
We assume that the value of $p_1$ under $h$ is $\frac{4}{9}$, and the value of $p_2$ under $h$ is $\frac{5}{9}$. 
We put $M=N\times S^1$ and we define $f=h\times \text{id}_{S^1}:M\rightarrow I\times S^1$. 
We denote by $Z_1\subset M$ (resp. $Z_2\subset M$) the component of indefinite folds of $f$ satisfying $f(Z_1)=\{\frac{4}{9}\}\times S^1$ (resp. $f(Z_2)=\{\frac{5}{9}\}\times S^1$). 

We identify $S^1$ with $[0,1]/\{0,1\}$. 
By construction of $N$, we can identify $f^{-1}(\{ \frac{1}{2} \} \times \{0\})$ with the closed surface $\Sigma_{g+1}$. 
Moreover, this identification is unique up to Dehn twist $t_{c}$, where $c\subset \Sigma_{g+1}$ is the belt sphere of the $1$-handle. 
We denote by $d\subset \Sigma_{g+1}$ the attaching circle of the $2$-handle. 
In this section, we look at a monodromy of the fibration $f$, especially how a monodromy along the curve $\gamma=\{\frac{1}{2}\}\times S^1$ is changed by a certain homotopy of $f$. 

We first remark that the number of connected components of the complement $\Sigma_{g+1}\setminus (c\cup d)$ is at most $2$. 
We call a pair $(c,d)$ a {\it bounding pair of genus-$g_1$} if the complement $\Sigma_{g+1}\setminus (c\cup d)$ consists of two twice punctured surfaces of genus $g_1$ and $g_2=g-g_1$.

Let $c,d \subset \Sigma_{g+1}$ be mutually disjoint non-separating simple closed curves. 
We look at details of the following homomorphisms: 
\begin{align*}
\Phi_{c} & :\MCG{(\Sigma_{g+1})(c,d)}\rightarrow \MCG{(\Sigma_g)}(d), \\
\Phi_{d} & :\MCG{(\Sigma_{g+1})(c,d)}\rightarrow \MCG{(\Sigma_g)}(c).
\end{align*}

We first consider the case that a pair $(c,d)$ is not a bounding pair. 
In this case, $c$ and $d$ are non-separating curves in $\Sigma_g$. 
As we mentioned in Section \ref{sec_preliminaries}, for a non-separating simple closed curve $c\subset \Sigma_g$, the homomorphism $\Phi_c$ is defined as $F_{v_1,v_2}\circ \Phi_c^{\ast}$. 
It is proved in \cite{Farb_Margalit} that the kernel of the homomorphism $\Phi_c^{\ast}$ is generated by the Dehn twist $t_{c}$.

Let $\MCG{(\Sigma_g)}(c^\text{ori})$ be the subgroup of $\MCG{(\Sigma_g)}(c)$ whose element is represented by a diffeomorphism preserving an orientation of $c$. 
We can define the homomorphism $\Phi_c^\text{ori} : \MCG{(\Sigma_g)}(c^\text{ori})\rightarrow \mathcal{M}_{g-1}$ as we define $\Phi_c$. 
Furthermore, we can decompose this map as follows: 
\[
\Phi_c^\text{ori} : \MCG{(\Sigma_{g})}(c^\text{ori}) \xrightarrow{\Phi_c^{\ast, \text{ori}}} \MCG{(\Sigma_{g-1}, v_1,v_2)} \xrightarrow{F_{v_1,v_2}}\mathcal{M}_{g-1}. 
\]
For $g\geq 3$, it is known that the kernel of the map $F_{v_1,v_2}:\MCG{(\Sigma_{g-1},v_1,v_2)} \rightarrow \mathcal{M}_{g-1}$ is isomorphic to the fundamental group of the configuration space $\Sigma_{g-1}\times \Sigma_{g-1}\setminus \Delta\Sigma_{g-1}$, where $\Delta\Sigma_{g-1}\subset \Sigma_{g-1}\times \Sigma_{g-1}$ is the diagonal set. 

We define the subgroups $\MCG{(\Sigma_{g+1})}(c^{\text{ori}},d)$, $\MCG{(\Sigma_{g+1})}(c,d^{\text{ori}})$ and $\MCG{(\Sigma_{g+1})}(c^{\text{ori}},d^{\text{ori}})$ of the group $\MCG{(\Sigma_{g+1})}(c,d)$ as we define the group $\MCG{(\Sigma_g)}(c^\text{ori})$. 
By the argument above, we obtain the following commutative diagram: 

\begin{equation}\label{keyCD}
\begin{xy}
{(0,0) *{\MCG{(\Sigma_{g+1})}(c^{\text{ori}},d^{\text{ori}})}},
{(20,20) *{\MCG{(\Sigma_{g},w_1,w_2)}(c^{\text{ori}})}},
{(-20,20) *{\MCG{(\Sigma_{g},v_1,v_2)}(d^{\text{ori}})}},
{(0,40) *{\MCG{(\Sigma_{g-1},v_1,v_2,w_1,w_2)}}},
{(20,-20) *{\MCG{(\Sigma_{g})}(c^{\text{ori}})}},
{(-20,-20) *{\MCG{(\Sigma_{g})}(d^{\text{ori}})}},
{(-50,-40) *{\hspace{3em}\MCG{(\Sigma_{g-1},w_1,w_2)}}},
{(50,-40) *{\MCG{(\Sigma_{g-1},v_1,v_2)}}},
{(0,-40) *{\mathcal{M}_{g-1}}},
{(8,4) \ar_{\Phi_d^{\ast,\text{ori}}} (20,16)}, 
{(-8,4) \ar^{\Phi_{c}^{\ast , \text{ori}}} (-20,16)}, 
{(8,-4) \ar^{\Phi_d^{\text{ori}}} (20,-16)}, 
{(-8,-4) \ar_{\Phi_{c}^{\text{ori}}} (-20,-16)}, 
{(20,24) \ar_{\Phi_{c}^{\ast,\text{ori}}} (8,36)}, 
{(-20,24) \ar^{\Phi_d^{\ast,\text{ori}}} (-8,36)}, 
{(18,-24) \ar_{\Phi_c^{\text{ori}}} (6,-36)}, 
{(-18,-24) \ar^{\Phi_d^{\text{ori}}} (-6,-36)}, 
{(28,-24) \ar^{\Phi_c^{\ast,\text{ori}}} (50,-36)}, 
{(-28,-24) \ar_{\Phi_d^{\ast,\text{ori}}} (-50,-36)}, 
{(-24,40) \ar_{F_{v_1,v_2}} @/_10mm/ (-60,-36)}, 
{(24,40) \ar^{F_{w_1,w_2}} @/^10mm/ (60,-36)}, 
{(-30,-40) \ar^{F_{w_1,w_2}} (-8,-40)}, 
{(30,-40) \ar_{F_{v_1,v_2}}  (8,-40)}, 
{(25,16) \ar^{F_{w_1,w_2}} @/^5mm/ (25,-16)}, 
{(-25,16) \ar_{F_{v_1,v_2}} @/_5mm/ (-25,-16)}, 
\end{xy}
\end{equation}
\vspace{.8em}

Since $c$ is disjoint from $d$, the kernel of the map $\Phi_{c}:\MCG{(\Sigma_{g+1})}(c,d)\rightarrow \MCG{(\Sigma_{g})}(d)$ is contained in the group $\MCG{(\Sigma_{g+1})}(c, d^{\text{ori}})$. 
Similarly, $\Ker{\Phi_d}$ is contined in $\MCG{(\Sigma_{g+1})}(c^{\text{ori}},d)$. 
Thus, we obtain: 
\[
\Ker{\Phi_{c}}\cap \Ker{\Phi_d} \subset \MCG{(\Sigma_{g+1})}(c^\text{ori}, d^{\text{ori}}), 
\]
and 
\[
\Ker{\Phi_{c}}\cap \Ker{\Phi_d}=\Ker{\Phi_{c}^\text{ori}}\cap \Ker{\Phi_d^\text{ori}}. 
\]
The map $\Phi_d^{\ast,\text{ori}}\circ\Phi_{c}^{\ast,\text{ori}}=\Phi_c^{\ast,\text{ori}}\circ\Phi_{d}^{\ast,\text{ori}}$ sends the group $\Ker{\Phi_{c}^\text{ori}}\cap \Ker{\Phi_d^\text{ori}}$ to the group $\Ker{F_{v_1,v_2}}\cap \Ker{F_{w_1,w_2}}\subset \MCG{(\Sigma_{g-1},v_1,v_2,w_1,w_2)}$, which is contained in the following group: 
\[
\Ker{(F_{v_1,v_2,w_1,w_2}: \MCG{(\Sigma_{g-1},v_1,v_2,w_1,w_2)}\rightarrow \mathcal{M}_{g-1})}. 
\]

\begin{lem}\label{lem_intersection1}

The following restrictions are isomorphic: 
\begin{align*}
\Phi_d^{\ast,\text{ori}}\circ\Phi_{c}^{\ast,\text{ori}} |_{\Ker{\Phi_{c}}\cap \Ker{\Phi_d}} & : \Ker{\Phi_{c}}\cap \Ker{\Phi_d}\rightarrow \Ker{F_{v_1,v_2}}\cap \Ker{F_{w_1,w_2}}, \\
\Phi_d^{\ast,\text{ori}} |_{\Ker{\Phi_{c}}\cap \Ker{\Phi_d}} & : \Ker{\Phi_{c}}\cap \Ker{\Phi_d}\rightarrow \Ker{\Phi_{c}^{\text{ori}}}\cap \Ker{F_{w_1,w_2}}, \\
\Phi_{c}^{\ast,\text{ori}} |_{\Ker{\Phi_{c}}\cap \Ker{\Phi_d}} & : \Ker{\Phi_{c}}\cap \Ker{\Phi_d}\rightarrow \Ker{F_{v_1,v_2}}\cap \Ker{\Phi_{d}^{\text{ori}}}, \\ 
\end{align*}

\end{lem}

\begin{proof}[Proof of Lemma \ref{lem_intersection1}]

We only prove that the first map is isomorphism (we can prove the other maps are isomorphic similarly). 
In this proof, we denote the map $\Phi_d^{\ast,\text{ori}}\circ\Phi_{c}^{\ast,\text{ori}} |_{\Ker{\Phi_{c}}\cap \Ker{\Phi_d}}$ by $\Phi$ for simplicity. 
We first prove that $\Phi$ is injective. 
We take an element $\psi\in\Ker{\Phi}$. 
Since the kernel of $\Phi_{c}^{\ast,\text{ori}}$ (resp. $\Phi_{d}^{\ast, \text{ori}}$) is generated by $t_{c}$ (resp. $t_d$), $\psi$ is equal to $t_{c}^{m}\cdot t_{d}^{n}$, for some $m,n\in\mathbb{Z}$. 
Since $\psi$ is contained in $\Ker{\Phi_{c}}$, we have $\Phi_{c}(\psi)=t_{d}^n=1$. 
Thus, we obtain $n=0$. 
Similarly, we can obtain $m=0$ and this completes the proof of injectivity of $\Phi$. 

We next prove that $\Phi$ is surjective. 
For an element $\xi\in \Ker{F_{v_1,v_2}}\cap \Ker{F_{w_1,w_2}}$, we can take an element $\overline{\xi}\in \MCG{(\Sigma_{g+1})}(c,d)$ which mapped to $\xi$ by the map $\Phi_d^{\ast,\text{ori}}\circ\Phi_{c}^{\ast,\text{ori}}$ since both of the maps $\Phi_d^{\ast,\text{ori}}$ and $\Phi_{c}^{\ast,\text{ori}}$ are surjective. 
By the commutative diagram (\ref{keyCD}), $\Phi_{d}^{\text{ori}}(\overline{\xi})$ is contained in the kernel of $\Phi_{c}^{\ast,\text{ori}}$. 
Thus, we obtain $\Phi_{d}^{\text{ori}}(\overline{\xi})={t_{c}}^n$, for some $n\in\mathbb{Z}$. 
Similarly, we obtain $\Phi_{c}^{\text{ori}}(\overline{\xi})={t_d}^m$, for some $m\in\mathbb{Z}$. 
Therefore, $\xi\cdot {t_{c}}^{-n}\cdot {t_d}^{-m}$ is contained in the group $\Ker{\Phi_{c}}\cap \Ker{\Phi_d}$ and mapped to $\xi$ by the map $\Phi$. 
This completes the proof of surjectivity of $\Phi$. 

\end{proof}

Let $\varepsilon: \Diff^+{\Sigma_{g-1}}\rightarrow {\Sigma_{g-1}}^4\setminus \tilde{\Delta}$ be the evaluation map at the points $v_1,v_2,w_1,w_2\in \Sigma_{g-1}$, where $\tilde{\Delta}$ is the subset of ${\Sigma_{g-1}}^4$ defined as follows: 
\[
\tilde{\Delta}=\{(x_1,x_2,x_3,x_4)\in{\Sigma_{g-1}}^4 \hspace{.3em} | \hspace{.3em} {}^\exists i\neq {}^\exists j \text{ s.t. }x_i=x_j\}. 
\]
Birman proved in \cite{Birman} that the map $\varepsilon$ is a locally trivial fibration with fiber $\Diff^+{(\Sigma_{g-1},v_1,v_2,w_1,w_2)}$. 
Since ${\Sigma_{g-1}}^4\setminus \tilde{\Delta}$ is connected, we obtain the following exact sequence: 
\begin{align}\label{Birmanexactsequence}
\pi_1(\Diff^+{(\Sigma_{g-1},v_1,v_2,w_1,w_2)}, \text{id})\rightarrow \pi_1(\Diff^+{\Sigma_{g-1}}, \text{id}) \xrightarrow{\varepsilon_\ast} \pi_1({\Sigma_{g-1}}^4\setminus \tilde{\Delta}, (v_1,v_2,w_1,w_2)) \\
\rightarrow \MCG{(\Sigma_{g-1}, v_1,v_2,w_1,w_2)} \rightarrow \mathcal{M}_{g-1} \rightarrow 1.  \nonumber  
\end{align}

Note that the map $\MCG{(\Sigma_{g-1}, v_1,v_2,w_1,w_2)} \rightarrow \mathcal{M}_{g-1}$ corresponds to the map $F_{v_1,v_2,w_1,w_2}$. 
Let $\Diff_0^+{\Sigma_{g-1}}$ be a connected component of $\Diff^+{\Sigma_{g-1}}$ which contains the identity map. 
The group $\Diff_0^+{\Sigma_{g-1}}$ is contractible if $g\geq 3$ (cf. \cite{Earle_Eells}). 
Thus, if $g\geq 3$, the kernel of the map $F_{v_1,v_2,w_1,w_2}$ is isomorphic to the fundamental group of the configuration space ${\Sigma_{g-1}}^4\setminus \tilde{\Delta}$.
Moreover, under the identification $\Ker{F_{v_1,v_2,w_1,w_2}}\cong \pi_1({\Sigma_{g-1}}^4\setminus \tilde{\Delta}, (v_1,v_2,w_1,w_2))$, the kernel of the map $F_{w_1,w_2}$ corresponds to the following homomorphism: 
\[
p_{1,\ast}:\pi_1({\Sigma_{g-1}}^4\setminus \tilde{\Delta},(v_1,v_2,w_1,w_2)) \rightarrow \pi_1(\Sigma_{g-1}\times \Sigma_{g-1}\setminus \Delta\Sigma_{g-1},(v_1,v_2)), 
\]
where $p_1$ is the projection onto the first and second components. 
Similarly, the kernel of the map $F_{v_1,v_2}$ corresponds to the following homomorphism: 
\[
p_{2,\ast}:\pi_1({\Sigma_{g-1}}^4\setminus \tilde{\Delta},(v_1,v_2,w_1,w_2)) \rightarrow \pi_1(\Sigma_{g-1}\times \Sigma_{g-1}\setminus \Delta\Sigma_{g-1},(w_1,w_2)), 
\]
where $p_2$ is the projection onto the third and fourth components. 
Eventually, we obtain the following isomorphism: 
\[
\Ker{F_{v_1,v_2}}\cap \Ker{F_{w_1,w_2}}\cong \Ker{p_{1,\ast}} \cap \Ker{p_{2,\ast}}. 
\]

For an oriented surface $S$ and points $x,y\in S$, we define $\Pi(S,x,y)$ as the set of embedded path from $x$ to $y$. 
For an element $\eta\in\Pi(S,x,y)$, we denote by $L(\eta):([0,1],\{0,1\})\rightarrow (S\setminus \{y\},x)$ a loop in the neighborhood of $\eta$, which is injective on $[0,1)$ and homotopic to a loop obtained by connecting $x$ to a sufficiently small counterclockwise circle around $y$ using $\eta$. 

\begin{lem}\label{lem_intersection2}

For an element $\eta\in\Pi(\Sigma_{g-1}\setminus\{v_i,w_j\},v_k,w_l)$ ($\{i,k\}=\{j,l\}=\{1,2\}$), we denote by $l(\eta)$ the following loop:
\[
[0,1]\ni t \mapsto \begin{cases}
(L(\eta)(t),v_2,w_1,w_2) & (k=1) \\
(v_1,L(\eta)(t),w_1,w_2) & (k=2) 
\end{cases}\in {\Sigma_g}^4\setminus \tilde{\Delta}. 
\]
Then, the group $\Ker{p_{1,\ast}} \cap \Ker{p_{2,\ast}}$ is generated by the following set: 
\[
\{[l(\eta)]\in\pi_1({\Sigma_{g-1}}^4\setminus \tilde{\Delta}, (v_1,v_2,w_1,w_2)) \hspace{.3em}|\hspace{.3em} \eta\in\Pi(\Sigma_{g-1}\setminus\{v_i,w_j\},v_k,w_l),  \{i,k\}=\{j,l\}=\{1,2\} \}.
\]

\end{lem}

\begin{proof}[Proof of Lemma \ref{lem_intersection2}]

When the space $S$ is obvious, we denote by $\Delta$ the diagonal subset of $S\times S$ for simplicity. 
It is obvious that an element $[l(\eta)]$ is contained in the group $\Ker{p_{1,\ast}}\cap \Ker{p_{2,\ast}}$ for any $\eta\in\Pi(\Sigma_{g-1}\setminus \{v_i,w_j\},v_k,w_l)$. 
We prove that any element of $\Ker{p_{1,\ast}}\cap \Ker{p_{2,\ast}}$ can be represented by the product $[l(\eta_1)\cdot \cdots [l(\eta_m)]$, for some $\eta_p\in \Pi(\Sigma_{g-1}\setminus \{v_{i_p},w_{j_p}\},v_{k_p},w_{l_p})$. 
To prove this, we need the following lemma. 

\begin{lem}[Theorem 3 of Fadell-Neuwirth \cite{Fadell_Neuwirth}]\label{lem_FadellNeuwirth}

The projection 
\[
p_2: {\Sigma_{g-1}}^4\setminus \tilde{\Delta} \rightarrow {\Sigma_{g-1}}^{2}\setminus \Delta
\]
is a locally trivial fibration with fiber $(\Sigma_{g-1}\setminus \{w_1,w_2\})^2\setminus \Delta$. 

\end{lem}

By Lemma \ref{lem_FadellNeuwirth}, we obtain the following homotopy exact sequence: 
\begin{align*}
\pi_2({\Sigma_{g-1}}^{2}\setminus \Delta, (w_1,w_2)) \rightarrow \pi_1((\Sigma_{g-1}\setminus \{w_1,w_2\})^2\setminus \Delta, (v_1,v_2)) \rightarrow \pi_1({\Sigma_{g-1}}^4\setminus \tilde{\Delta}, (v_1,v_2,w_1,w_2)) \\
\xrightarrow{p_{2,\ast}} \pi_1({\Sigma_{g-1}}^{2}\setminus \Delta, (w_1,w_2)) \rightarrow \pi_0((\Sigma_{g-1}\setminus \{w_1,w_2\})^2\setminus \Delta, (v_1,v_2)).  
\end{align*}

Since the space $(\Sigma_{g-1}\setminus \{w_1,w_2\})^2\setminus \Delta$ is connected and the space ${\Sigma_{g-1}}^{2}\setminus \Delta$ is aspherical (cf. Corollary 2.2. of \cite{Fadell_Neuwirth}), the inclusion map $i:(\Sigma_{g-1}\setminus \{w_1,w_2\})^2\setminus \Delta \rightarrow {\Sigma_{g-1}}^4\setminus \tilde{\Delta}$ gives the following isomorphism: 
\[
i_{\ast}: \pi_1((\Sigma_{g-1}\setminus \{w_1,w_2\})^2\setminus \Delta, (v_1,v_2)) \rightarrow \Ker{p_{2,\ast}}. 
\] 
Let $i^\prime :(\Sigma_{g-1}\setminus \{w_1,w_2\})^2\setminus \Delta \rightarrow {\Sigma_{g-1}}^{2}\setminus \Delta$ be the inclusion map. 
The group $\Ker{p_{1,\ast}}\cap \Ker{p_{2,\ast}}$ is isomorphic to the group $\Ker{i^\prime_{\ast}}$ since the following diagram commutes: 

\begin{center}
\begin{minipage}[c]{100mm}
\begin{xy}
{(0,0) *{(\Sigma_{g-1}\setminus \{w_1,w_2\})^2\setminus \Delta}}, 
{(40,0) *{{\Sigma_{g-1}}^4\setminus \tilde{\Delta}}}, 
{(0,-15) *{{\Sigma_{g-1}}^{2}\setminus \Delta}}, 
{(20,0) \ar ^{i} (30,0)}, 
{(0,-4), \ar _{i^\prime} (0,-11)}, 
{(32,-4) \ar ^{p_1} (12, -12)}
\end{xy}
\end{minipage}
\end{center}

\noindent
Thus, it is sufficient to prove that any element of $\Ker{i^\prime_{\ast}}$ can be represented by the product $[l^\prime(\eta_1)\cdot \cdots \cdot  [l^\prime(\eta_m)]$ for some $\eta_p\in\Pi(\Sigma_{g-1}\setminus \{v_{i_p},w_{j_p}\},v_{k_p},w_{l_p})$, where $l^\prime(\eta_p)$ is the loop defined as follows: 
\[
[0,1]\ni t \mapsto \begin{cases}
(L(\eta_p)(t),v_2) & (k_p=1) \\
(v_1,L(\eta_p)(t)) & (k_p=2) 
\end{cases}\in (\Sigma_{g-1}\setminus \{w_1,w_2\})^2\setminus \Delta. 
\]

We take an element $[\xi]\in \Ker{p_{1,\ast}}$, where $\xi:(S^1,1)\rightarrow ((\Sigma_{g-1}\setminus\{w_1,w_2\})^2\setminus\Delta, (v_1,v_2))$ is a loop ($1\in S^1\subset \mathbb{C}$). 
We can assume that $\xi$ is an embedding. 
Since $\xi$ is null-homotopic in the space ${\Sigma_{g-1}}^2\setminus \Delta$, we can take a map $\overline{\xi}: D^2 \rightarrow {\Sigma_{g-1}}^2\setminus \Delta$ satisfying the following conditions: 

\begin{enumerate}[(a)]

\item the restriction $\res{\overline{\xi}}: S^1=\partial D^2 \rightarrow {\Sigma_{g-1}}^2\setminus \Delta$ corresponds to $\xi$, 

\item $\overline{\xi}$ is a complete immersion, that is, $\overline{\xi}$ satisfies: 

\begin{itemize}

\item $\overline{\xi}$ is an immersion, 

\item $\sharp \overline{\xi}^{-1}(p)$ is at most $2$ for each $p\in \overline{\xi}(D^2)$, 

\item for any point $p\in \overline{\xi}(D^2)$ such that $\sharp \overline{\xi}^{-1}(p)=2$, there exists a disk neighborhood $D_i\subset {\Sigma_{g-1}}^2\setminus \Delta$ of a point $p_i\in \overline{\xi}^{-1}(p)$ such that $\overline{\xi}$ is an embedding over $D_i$, and that $\overline{\xi}(D_1)$ intersects $\overline{\xi}(D_2)$ at the unique point $p$ transversely, where $\{p_1,p_2\}=\overline{\xi}^{-1}(p)$, 

\end{itemize}

\item for each $i\in\{1,2\}$, $\overline{\xi}^{-1}\Bigl(\{w_i\}\times (\Sigma_{g-1}\setminus \{w_i\})\Bigr)$ and $\overline{\xi}^{-1}\Bigl((\Sigma_{g-1}\setminus \{w_i\})\times \{w_i\}\Bigr)$ is a discrete set and is contained in $\Int{D^2} \cap \mathbb{R}$, 

\item the set $\overline{\xi}^{-1}\Bigl(\{p\in{\Sigma_{g-1}}^2\setminus \Delta \hspace{.3em} | \hspace{.3em} \sharp \overline{\xi}(p)=2\}\Bigr)$ is contained in $\Int{D^2} \cap \mathbb{R}$, 

\item $\overline{\xi}(D^2)$ does not contain the point $(w_1,w_2)$ and $(w_2,w_1)$.  

\end{enumerate}

We define a discrete set $B\subset \Int{D^2}\cap \mathbb{R}$ as follows:  
\[
B=\coprod_{i=1}^{2}\overline{\xi}^{-1}\Bigl(\{w_i\}\times (\Sigma_{g-1}\setminus \{w_i\}\Bigr) \coprod_{j=1}^{2} \overline{\xi}^{-1}\Bigl(\Sigma_{g-1}\setminus \{w_j\})\times \{w_j\}\Bigr) \cup \overline{\xi}^{-1}\Bigl(\{p\in{\Sigma_{g-1}}^2\setminus \Delta \hspace{.3em} | \hspace{.3em} \sharp \overline{\xi}(p)=2\}\Bigr). 
\]
We put $B=\{q_1,\ldots,q_n\} \subset D^2\cap \mathbb{R}$.
We assume that $q_1< \cdots < q_n$. 
Denote by $S_i$ the upper semicircle centered at $\frac{1+q_i}{2}$ whose ends are $1$ and $q_i$. 
We also denote by $\zeta_i$ a loop obtained by connecting a small counterclockwise circle around $q_i$ to the point $1\in S^1$ using $S_i$. 
Since $\overline{\xi}$ is an embedding over $S_i$, the image $\overline{\xi}(S_i)$ is an embedded path, which we denote by $(\eta_1(S_i),\eta_2(S_i))\subset {\Sigma_{g-1}}^2\setminus \Delta$. 
The loop $\overline{\xi}(\zeta_i)$ is homotopic to one of the following loops: 
\begin{equation*}
\overline{\xi}(\zeta_i) \simeq \begin{cases}
l^\prime(\eta_1(S_i)) & \text{(if $\overline{\xi}(q_i)$ is contained in $\{w_i\}\times (\Sigma_{g-1}\setminus \{w_i\}$)}, \\
l^\prime(\eta_2(S_i)) & \text{(if $\overline{\xi}(q_i)$ is contained in $(\Sigma_{g-1}\setminus \{w_j\})\times \{w_j\}$)}, \\
\text{trivial loop} & \text{(otherwise)}. 
\end{cases}
\end{equation*}
The loop $\xi$ is homotopic to the loop $\overline{\xi}|_{\zeta_1\cdot \cdots \cdot \zeta_n}$, and this completes the proof of Lemma \ref{lem_intersection2}.  

\end{proof}

We eventually obtain the following theorem. 

\begin{thm}\label{keythm_intersection}

For an element $\eta\in\Pi(\Sigma_{g-1}\setminus \{v_i,w_j\}, v_k,w_l)$ ($\{i,k\}=\{j,l\}=\{1,2\}$), we denote by $\delta(\eta)\subset \Sigma_{g-1}$ the boundary of a regular neighborhood of $\eta$. 
This is a simple closed curve in $\Sigma_{g-1}\setminus \{v_1,v_2,w_1,w_2\}$ and we can take a lift of this curve to $\tilde{\delta}(\eta)\subset \Sigma_{g+1}\setminus (c\cup d)$ by using the identification $\Sigma_{g-1}\setminus \{v_1,v_2,w_1,w_2\}\cong \Sigma_{g+1}\setminus (c\cup d)$. 
If $g$ is greater than $2$, then the group $\Ker{\Phi_{c}}\cap \Ker{\Phi_d}$ is generated by the following set: 
\[
\{t_{\tilde{\delta}(\eta)}\cdot t_{c}^{-1}\cdot t_{d}^{-1}\in \MCG{(\Sigma_{g+1})}(c,d) \hspace{.3em} | \hspace{.3em} \eta\in\Pi(\Sigma_{g-1}\setminus \{v_i,w_j\}, v_k,w_l), \{i,k\}=\{j,l\}=\{1,2\}\}. 
\]

\end{thm}

We next consider the case $(c,d)$ is a bounding pair of genus $g_1$. 
Then, $c\subset \Sigma_g$ is a separating curve. 
We put $g_2=g-g_1$. 
By the same argument as in Lemma \ref{lem_intersection1}, we can prove the following lemma. 

\begin{lem}\label{lem_intersection1_separating}

The following restrictions are isomorphic: 
\begin{align*}
\Phi_d^{\ast}\circ\Phi_{\tilde{c}}^{\ast} |_{\Ker{\Phi_{c}}\cap \Ker{\Phi_d}} & : \Ker{\Phi_{c}}\cap \Ker{\Phi_d}\rightarrow \Ker{F_{v_1,v_2}}\cap \Ker{F_{w_1,w_2}}, \\
\Phi_d^{\ast} |_{\Ker{\Phi_{c}}\cap \Ker{\Phi_d}} & : \Ker{\Phi_{c}}\cap \Ker{\Phi_d}\rightarrow \Ker{\Phi_{c}}\cap \Ker{F_{w_1,w_2}}, \\
\Phi_{\tilde{c}}^{\ast} |_{\Ker{\Phi_{c}}\cap \Ker{\Phi_d}} & : \Ker{\Phi_{c}}\cap \Ker{\Phi_d}\rightarrow \Ker{F_{v_1,v_2}}\cap \Ker{\Phi_{d}}, \\ 
\end{align*}

\end{lem}

The group $\Ker{F_{v_1,v_2}}$ (resp. $\Ker{F_{w_1,w_2}}$) corresponds to the group $\Ker{F_{v_1}}\times \Ker{F_{v_2}}$ (resp. $\Ker{F_{w_1}}\times \Ker{F_{w_2}}$). 
Thus, we obtain: 
\[
\Ker{F_{v_1,v_2}}\cap\Ker{F_{w_1,w_2}}= (\Ker{F_{v_1}}\cap \Ker{F_{w_1}})\times (\Ker{F_{v_2}}\cap \Ker{F_{w_2}}).  
\]
Furthermore, the group $\Ker{F_{v_i}}\cap \Ker{F_{w_i}}$ is contained in the kernel of the following homomorphism:
\[
F_{v_i,w_i}:\MCG{(\Sigma_{g_i},v_i, w_i)}\rightarrow \mathcal{M}_{g_i}.
\] 
This group is isomorphic to the group $\pi_1({\Sigma_{g_i}}^2\setminus \Delta, (v_i,w_i))$ if $g_i\geq 2$. 
Under this identification, it is easy to prove that $\Ker{F_{v_i}}\cap \Ker{F_{w_i}}$ corresponds to the group $\Ker{p_{1,\ast}}\cap \Ker{p_{2,\ast}}$. 
where we denote by $p_j: {\Sigma_{g_i}}^2\setminus \Delta \rightarrow \Sigma_{g_i}$ the projection onto the $j$-th component. 
Since $p_2$ is a locally trivial fibration with fiber $\Sigma_{g_i}\setminus \{w_i\}$ (cf. \cite{Fadell_Neuwirth}), we can prove the following lemma by using Van Kampen's theorem. 

\begin{lem}\label{lem_intersection2_separating}

For an element $\eta\in\Pi(\Sigma_{g_i},v_i, w_i)$, we denote by $l(\eta)$ the following loop:
\[
[0,1]\ni t \mapsto (L(\eta)(t),w_i) \in {\Sigma_{g_i}}^2\setminus \Delta. 
\]
Then, the group $\Ker{p_{1,\ast}} \cap \Ker{p_{2,\ast}}$ is generated by the following set: 
\[
\{[l(\eta)]\in\pi_1({\Sigma_{g_i}}^2\setminus \Delta, (v_i,w_i)) \hspace{.3em}|\hspace{.3em} \eta\in\Pi(\Sigma_{g_i},v_i,w_i) \}.
\]

\end{lem}

As the case $(c,d)$ is not a bounding pair, we eventually obtain the following theorem. 

\begin{thm}\label{keythm_intersection_separating}

For an element $\eta\in\Pi(\Sigma_{g_i}, v_i,w_i)$, we denote by $\delta(\eta)\subset \Sigma_{g_i}$ the boundary of a regular neighborhood of $\eta$. 
This is a simple closed curve in $\Sigma_{g_i}\setminus \{v_i,w_i\}$ and we can take a lift of this curve to $\tilde{\delta}(\eta)\subset \Sigma_{g_1+g_2+1}\setminus (c\cup d)$ by using the identification $\Sigma_{g_1}\setminus \{v_1,w_1\}\amalg \Sigma_{g_2}\setminus \{v_2,w_2\}\cong \Sigma_{g_1+g_2+1}\setminus (c\cup d)$. 
If both of the numbers $g_1$ and $g_2$ are greater than or equal to $2$, then the group $\Ker{\Phi_{c}}\cap \Ker{\Phi_d}$ is generated by the following set: 
\[
\{t_{\tilde{\delta}(\eta)}\cdot t_{c}^{-1}\cdot t_{d}^{-1}\in \MCG{(\Sigma_{g+1})}(c,d) \hspace{.3em} | \hspace{.3em} \eta\in\Pi(\Sigma_{g_i}, v_i,w_i), i \in\{1,2\}\}. 
\]

\end{thm}

We are now ready to discuss the fibration $f:M\rightarrow I\times S^1$ which we defined in the beginning of this section. 
Let $N(p_i)\subset N$ be an open neighborhood of $p$ in $N$. 
We take a diffeomorphism $\theta_i: B_{\frac{1}{\sqrt{3}}} \rightarrow\nu N(p_i)$, where $B_{\frac{1}{\sqrt{3}}}\subset\mathbb{R}^3$ is a $3$-ball with radius $\frac{1}{{\sqrt3}}$, so that $h\circ \theta_i$ is described as follows: 
{\allowdisplaybreaks
\begin{align*}
\begin{array}{rccc}
h\circ\theta_1 : & B_{\frac{1}{\sqrt{3}}} & \longrightarrow & I  \\
& \rotatebox{90}{$\in$} &  & \rotatebox{90}{$\in$} \\ 
& (x,y,z) & \longmapsto     & x^2+y^2-z^2+\frac{4}{9},   
\end{array}\\
\begin{array}{rccc}
h\circ\theta_2 : & B_{\frac{1}{\sqrt{3}}} & \longrightarrow & I  \\
& \rotatebox{90}{$\in$} &  & \rotatebox{90}{$\in$} \\ 
& (x,y,z) & \longmapsto     & x^2-y^2-z^2+\frac{5}{9}.   
\end{array}
\end{align*}
}
We take a metric $g$ of $N$ so that the pull back $\theta_i^{\ast}g$ corresponds to the standard metric on $B_{\frac{1}{\sqrt{3}}}$. 
The metric $g$ determines a rank $1$ horizontal distribution $\mathcal{H}_h=(\Ker{dh})^{\perp}$ of $h|_{N\setminus \{p_1,p_2\}}$. 
For each $p\in N\setminus \{p_1,p_2\}$, we denote by $c_p(t)$ a horizontal lift of the curve $t\mapsto h(p)+t$ which satisfies $c_p(0)=p$. 
We define submanifolds $D_l^{\mathcal{H}_h}(p_i)$ and $D_u^{\mathcal{H}_h}(p_i)$ as follows: 
{\allowdisplaybreaks
\begin{align*}
D_l^{\mathcal{H}_h}(p_i) = \{p_i\} \cup \{p\in N \hspace{.3em} | \hspace{.3em} h(p)< \frac{3+i}{9}, \lim_{t \to\frac{3+i}{9}- h(p)} c_p(t)=p_i\}, \\
D_u^{\mathcal{H}_h}(p_i) = \{p_i\} \cup \{p\in N \hspace{.3em} | \hspace{.3em} h(p)> \frac{3+i}{9}, \lim_{t \to \frac{3+i}{9}- h(p)} c_p(t) =p_i\}. 
\end{align*}
}
Note that $D_l^{\mathcal{H}_h}(p_1)$ and $D_u^{\mathcal{H}_h}(p_2)$ are diffeomorphic to the unit interval $I$, while $D_u^{\mathcal{H}_h}(p_1)$ and $D_l^{\mathcal{H}_h}(p_2)$ are diffeomorphic to the $2$-disk $D^2$. 
We take a homotopy $h_t: N\rightarrow I$ with $h_0=h$ ($t\in I$) satisfying the following conditions: 

\begin{enumerate}[(a)]

\item the support of the homotopy is contained in $N(p_1)$, 

\item for any $t\in I$, $h_t$ has two critical points $p_1$ and $p_2$, 

\item for any $t\in I$, the critical point $p_1$ of $h_t$ is non-degenerate and the index of this is $1$, 

\item a function $t\mapsto h_t(p_1)$ is monotone increasing, 


\item $h_{1}(p_1)=\frac{2}{3}$. 

\end{enumerate}
This homotopy changes the order of critical points. 
We take a smooth function $\rho:I\rightarrow I$ satisfying the following properties: 

\begin{itemize}

\item $\rho\equiv 0$ on $\left[0,\frac{1}{6}\right]\amalg\left[\frac{5}{6},1\right]$, 

\item $\rho\equiv 1$ on $\left[\frac{1}{3}, \frac{2}{3}\right]$, 

\item $\rho$ is monotone increasing on $\left[\frac{1}{6}, \frac{1}{3}\right]$

\item $\rho(1-s)=\rho(s)$ for any $s\in [0,1]$. 

\end{itemize}

By using $h_t$ and $\rho$, we define a homotopy $f_t:M=N\times S^1\rightarrow I\times S^1$ as follows: 
\[
\begin{array}{rccc}
f_t : & M=N\times S^1 & \longrightarrow & I\times S^1  \\
& \rotatebox{90}{$\in$} &  & \rotatebox{90}{$\in$} \\ 
& (x,s) & \longmapsto     & (h_{t\rho(s)}(x), s). 
\end{array}
\]

Since $N$ is obtained by attaching the $1$-handle and the $2$-handle to $\Sigma_g\times I$, $\partial N$ contains the surface $\Sigma_g\times \{0\}$, which we denote by $\Sigma$ for simplicity. 
Moreover, $\Sigma$ intersects $D_l^{\mathcal{H}_{h}}(p_1)$ at two points $v_1,v_2\in \Sigma$, and $\Sigma$ intersects $D_l^{\mathcal{H}_h}(p_2)$ at a simple closed curve $d\subset\Sigma$. 
Let $\Pi(\Sigma, v_i, d)$ be a set of embedded paths from the point $v_i$ to a point in $d$. 
For an element $\eta\in \Pi(\Sigma,v_i,d)$, we denote by $L(\eta): ([0,1], \{0,1\})\rightarrow (\Sigma\setminus d, v_i)$ a loop in the neighborhood of $\eta\cup d$, which is injective on $[0,1)$ and homotopic to a loop obtained by connecting $v_i$ to $d$ using $\eta$. 
For an element $\eta\in \Pi(\Sigma,v_i,d)$, we take a homotopy of horizontal distributions $\{\mathcal{H}_t^{\eta}\}$ ($t\in[0,1]$) of $h_1|_{N\setminus \{p_1,p_2\}}$ with $\mathcal{H}_0^{\eta}=\mathcal{H}_{h_1}$ which satisfies the following conditions: 

\begin{enumerate}[(a)]

\setcounter{enumi}{5}

\item the support of the homotopy is contained in $h_1^{-1}(\left[\frac{5}{9},\frac{2}{3}\right])$, 

\item $\mathcal{H}_0^{\eta}=\mathcal{H}_1^{\eta}$, 

\item the arc $D_l^{\mathcal{H}_t^{\eta}}(p_1)$ intersects $\Sigma_g$ at the point $L(\eta)(t), v_j\in \Sigma$, where $\{i,j\}=\{1,2\}$. 

\end{enumerate}

\noindent
Such a homotopy exists because $L(\eta)$ is null-homotopic on the surface obtained by performing a surgery to $\Sigma$ along $d$. 

We next take a $1$-parameter family of homotopies $h_{t,s}: N\rightarrow I$ ($t,s\in I$) with $h_{0,s}=h_{\rho(s)}$ which satisfies the following conditions: 

\begin{enumerate}[(a)]

\setcounter{enumi}{8}

\item for any $s\in\left[0,\frac{1}{3}\right]\cup \left[\frac{2}{3},1\right]$, the homotopy $h_{t,s}$ corresponds to $h_{\rho(s)(1-t)}$, 

\item for any $t,s\in I$, $h_{t,s}$ has two critical points $p_1$ and $p_2$,

\item for any $s\in\left[\frac{1}{3},\frac{2}{3}\right]$, the support of the homotopy $h_{t,s}$ is contained in a small neighborhood of $D_{l}^{\mathcal{H}_{3s-1}^{\eta}}(p_1)\cup D_{u}^{\mathcal{H}_{3s-1}^{\eta}}(p_1)$, 
 
\item for any $s\in I$, the homotopy $h_{t,s}$ is identical in a neighborhood of $\partial N$, 

\item for any $t,s\in I$, the critical point $p_1$ of $h_{t,s}$ is non-degenerate and the index of this is $1$, 

\item for any $s\in I$, $h_{t,s}(p_1)$ is equal to $h_{\rho(s)(1-t)}(p_1)$, 

\end{enumerate}

\noindent
By using this family of homotopies, we define a homotopy $\tilde{f}_t:M\rightarrow I\times S^1$ as follows: 
\[
\begin{array}{rccc}
\tilde{f}_t : & M=N\times S^1 & \longrightarrow & I\times S^1  \\
& \rotatebox{90}{$\in$} &  & \rotatebox{90}{$\in$} \\ 
& (x,s) & \longmapsto     & (h_{t,s}(x), s), 
\end{array}
\]

Eventually, we obtain a new fibration $\tilde{f}_1$. 
By construction, $\tilde{f}_1$ can be obtained from the original fibration $f$ by the homotopies $f_t$ and $\tilde{f}_t$. 
In these homotopies, the image of singular loci are changed like Reidemeister move of type II (cf. Figure \ref{changesingularloci}). 
As is called in \cite{Wil2}, we call this kind of move an {\it $R_2$-move}. 

\begin{figure}[htbp]
\begin{center}
\includegraphics[width=110mm]{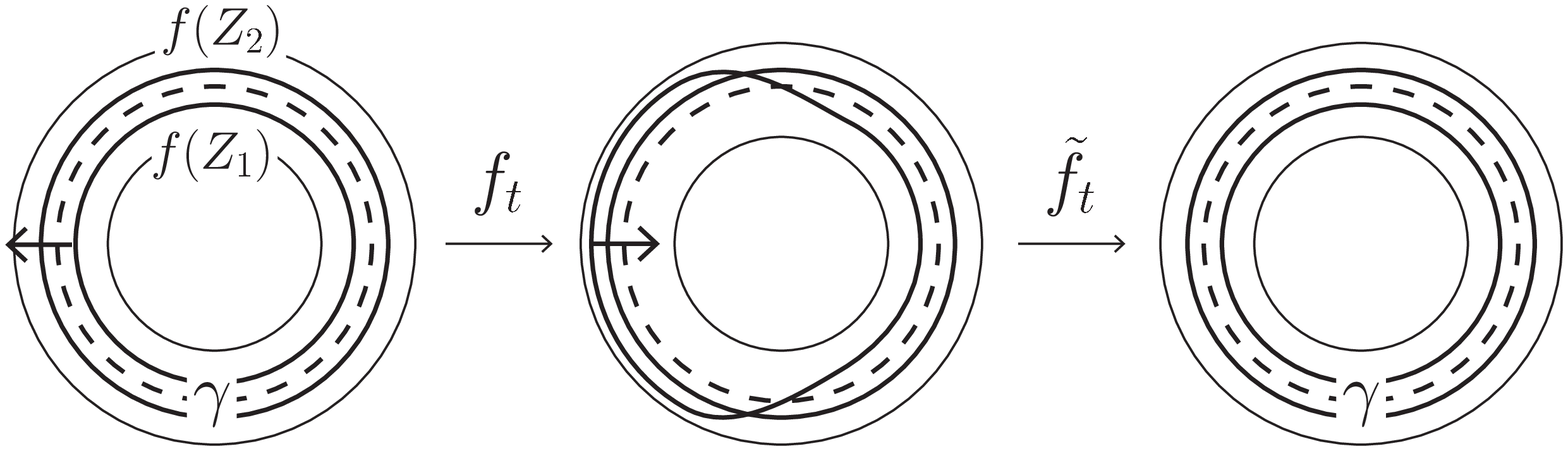}
\end{center}
\caption{Left: the image of singular loci of $f=f_0$. The bold circles describe the image $f(Z_1)\amalg f(Z_2)$ and the bold dotted circle describes $\gamma$. 
Center: the image of singular loci of $f_1=\tilde{f}_0$. 
Right: the image of singular loci of $\tilde{f}_1$, which corresponds to that of $f$. }
\label{changesingularloci}
\end{figure}

As mentioned in the beginning of this section, we can identify $f^{-1}(\{ \frac{1}{2} \} \times \{0\})$ with the closed surface $\Sigma_{g+1}$. 
Thus, a monodromy $\varphi_{\gamma}\in \mathcal{M}_{g+1}$ of $\tilde{f}$ along $\gamma$ can be defined. 
Since $\varphi$ is contained in the group $\MCG{(\Sigma_{g+1})}(c,d)$ and , an identification $f^{-1}(\{ \frac{1}{2} \} \times \{0\})\cong \Sigma_{g+1}$ is unique up to Dehn twist $t_{c}$, $\varphi_{\gamma}$ is independent of an identification $f^{-1}(\{ \frac{1}{2} \} \times \{0\})\cong \Sigma_{g+1}$. 

\begin{lem}\label{lem_monodromyalonggamma}

$\varphi_{\gamma}=t_{\tilde{\delta}(\eta)}\cdot t_{c}^{-1}\cdot t_{d}^{-1}$, where $\tilde{\delta}\subset \Sigma_{g+1}$ is a simple closed curve which corresponds to a regular neighborhood of $\eta\cup d\subset \Sigma$ under the identification $\Sigma\setminus \{v_1,v_2\}\cong \Sigma_{g+1}\setminus d$. 

\end{lem}

\begin{proof}[Proof of Lemma \ref{lem_monodromyalonggamma}]

Since both sides of the boundary $\partial M$ are trivial surface bundles, $\varphi_{\gamma}$ is contained in the group $\Ker{\Phi_c} \cap \Ker{\Phi_d}$. 
We consider the element $\Phi_c^{\ast}(\varphi_\gamma)\in \MCG{(\Sigma_g,v_1,v_2)}(d)$. 
This element can be realized as a monodromy of a certain fibration in the following way: 
we first take a sufficiently small neighborhood of the following subset of $M$: 
\[
\coprod_{s\in\left[0,\frac{1}{3}\right]\amalg\left[\frac{2}{3},1\right]} \biggl(\bigl(D_l^{\mathcal{H}_{h_{\rho(s)}}}(p_1)\cup D_u^{\mathcal{H}_{h_{\rho(s)}}}(p_1)\bigr)\times \{s\}\biggr) \coprod_{s\in\left[\frac{1}{3},\frac{2}{3}\right]} \biggl(\bigl(D_l^{\mathcal{H}_{3s-1}^{\eta}}(p_1)\cup D_u^{\mathcal{H}_{3s-1}^{\eta}}(p_1)\bigr)\times \{s\}\biggr).
\]
We denote this neighborhood by $U\subset M$. 
The restriction $\tilde{f}_0|_{M\setminus U}$ is a fibration with a connected singular locus $Z_2$. 
We take a suitable $U$ so that we can take a horizontal distribution $\tilde{\mathcal{H}}$ of $\tilde{f}_0|_{M\setminus (U\cup Z_2)}$ satisfying the following conditions:

\begin{itemize}

\item $\tilde{\mathcal{H}}$ is along the boundary $\partial U$, 

\item $\tilde{\mathcal{H}}$ corresponds to $\coprod\hspace{-1.6em}\raisebox{-1em}{\footnotesize $s\in S^1 $}(\mathcal{H}_{h_{\rho(s)}}\oplus T_s S^1)$ on a small neighborhood of $\subset M$, $\partial N\times S^1\subset M$ and $\tilde{f}_0^{-1}(I\times (\left[0, \frac{1}{6}\right]\cup \left[\frac{5}{6},1\right]))\subset M$. 

\end{itemize}

This distribution gives a monodromy of $\tilde{f}_0|_{M\setminus \overline{U}}$ along $\gamma$. 
We identify $\Sigma=\Sigma_g\times \{0\}\subset \partial N$ with $\Sigma_g$. 
The fiber $\tilde{f}_0^{-1}(\{\frac{1}{2}\}\times \{0\})\setminus \overline{U}$ is canonically identified with $\Sigma_g\setminus \{v_1,v_2\}$. 
By the condition (k) on the family of homotopies $\{h_{t,s}\}$, this monodromy corresponds to the element $\Phi_c^{\ast}(\varphi_\gamma)$. 

Since the region $\left[0,\frac{1}{2}\right]\times S^1$ does not contains any singular values of the fibration $\tilde{f}_0|_{M\setminus U}$, $\Phi_c^{\ast}(\varphi_\gamma)$ corresponds to the monodromy of $\tilde{f}_0|_{M\setminus U}$ along the following loop: 
\begin{equation*}
\tilde{\gamma}: I\ni t\mapsto \begin{cases}
(0,t) & (t\in \left[0,\frac{1}{3}\right]) \\
\left(\frac{9}{2}\left(t-\frac{1}{3}\right), \frac{1}{3}\right) & (t\in \left[\frac{1}{3},\frac{4}{9}\right]) \\
\left(\frac{1}{2}, 3\left(t-\frac{1}{3}\right)\right) & (t\in \left[\frac{4}{9},\frac{5}{9}\right]) \\
\left(\frac{9}{2}\left(\frac{2}{3}-t\right), \frac{2}{3}\right) & (t\in \left[\frac{5}{9},\frac{2}{3}\right]) \\
\left(0,t\right) & (t\in \left[\frac{2}{3},1\right]) \\
\end{cases}\in I\times S^1. 
\end{equation*}
We denote by $\psi_t: \tilde{f}_0^{-1}(\tilde{\gamma}(0))\cong \Sigma_g \rightarrow \tilde{f}_0^{-1}(\tilde{\gamma}(t))$ the diffeomorphism obtained by using the distribution $\tilde{\mathcal{H}}$ and the path $\tilde{\gamma}|_{[0,t]}$. 
Note that we can canonically identify $\tilde{f}_0^{-1}(\tilde{\gamma}(t))$ with $\Sigma_g$ for $t\in\left[0,\frac{1}{3}\right]\amalg\left[\frac{2}{3},1\right]$.  Moreover, under the identification, $\psi_t$ corresponds to the identity for $t\in\left[0,\frac{1}{3}\right]$, and $\psi_t=\psi_1$ for $t\in\left[\frac{2}{3},1\right]$ since $\tilde{\mathcal{H}}$ corresponds to $\coprod\hspace{-1.6em}\raisebox{-1em}{\footnotesize $s\in S^1 $}(\mathcal{H}_{h_{\rho(s)}}\oplus T_s S^1)$ on $\partial N\times S^1$. 

We can take the following diffeomorphism by using the horizontal distribution $\tilde{\mathcal{H}}$ of $\tilde{f}_0|_{M\setminus Z_1\cup Z_2}$ together with its horizontal lifts of $t\mapsto (t,s)\in I\times S^1$: 
\[
\tilde{\psi}_s: \Sigma_g\cong \tilde{f}_{0}^{-1}((0,s)) \rightarrow \tilde{f}_0^{-1}\left(\left(\frac{1}{2}, s\right)\right) \hspace{.5em}\left(s\in \left[\frac{1}{3},\frac{2}{3}\right]\right). 
\]
By the definitions of $\psi_t$ and $\tilde{\psi}_s$, we obtain the following equalities: 
{\allowdisplaybreaks
\begin{align*}
\tilde{\psi}_{\frac{1}{3}}^{-1}\circ \psi_{\frac{4}{9}} & = \text{id}_{\Sigma_g}, \\
\tilde{\psi}_{\frac{2}{3}}^{-1}\circ \psi_{\frac{5}{9}} & = \psi_1, \\
\tilde{\psi}_{3\left(t-\frac{1}{3}\right)}^{-1}\circ \psi_{t} (v_i) & = L(\eta)(9t-4) \hspace{.5em}\left(\text{for }t\in \left[\frac{4}{9},\frac{5}{9}\right]\right). 
\end{align*}
}The above equations mean that the path $[0,1]\ni t\mapsto \tilde{\psi}_{\frac{1}{3}(t+1)}^{-1}\circ \psi_{\frac{1}{9}(t+4)}\in \Diff^+(\Sigma_g, v_j)$ is the lift of the loop $L(\eta)$ in $\Sigma_g\setminus \{v_j\}$ under the following locally trivial fibration: 
\[
\Diff^+{(\Sigma_g, v_i, v_j)} \hookrightarrow \Diff^+{(\Sigma_g,v_j)} \xrightarrow{\varepsilon} \Sigma_g\setminus \{v_j\}, 
\] 
where $\varepsilon$ is the evaluation map. 
Thus, we obtain: 
\[
\Phi_{c}^{\ast}(\varphi)=[\psi_1]=Push(L(\eta)) = t_{\tilde{\delta}(\eta)}\cdot t_d^{-1}\in \MCG{(\Sigma_g, v_1,v_2)}(d), 
\]
where $Push(L(\eta))$ is the pushing map along $L(\eta)$. 
By Lemma \ref{lem_intersection1} or Lemma \ref{lem_intersection1_separating}, $\Phi_{c}^{\ast}|_{\Ker{\Phi_c}\cap \Ker{\Phi_d}}$ is an isomorphism. 
We therefore obtain: 
{\allowdisplaybreaks
\begin{align*}
\varphi_\gamma & =\Phi_{c}^{\ast, -1}\circ\Phi_{c}^{\ast}(\varphi_\gamma) \\
& =\Phi_c^{\ast, -1}(t_{\tilde{\delta}(\eta)}\cdot t_d^{-1}) \\
& =t_{\tilde{\delta}(\eta)}\cdot t_c^{-1}\cdot t_d^{-1}. 
\end{align*}
}
This completes the proof of Lemma \ref{lem_monodromyalonggamma}. 

\end{proof}

Combining Theorem \ref{keythm_intersection} and Theorem \ref{keythm_intersection_separating}, we obtain the following theorem. 

\begin{thm}\label{keythm_monodromyalonggamma}

Let $f:M\rightarrow I\times S^1$ and $\gamma\subset I\times S^1$ be as in the beginning of this section. 
Assume that $g$ is greater than or equal to $3$ when $(c,d)$ is not a bounding pair, and that both of the number $g_1$ and $g_2=g-g_1$ are greater than or equal to $2$ when a $(c,d)$ is a bounding pair of genus $g_1$. 
For any $\varphi\in \Ker{\Phi_c}\cap \Ker{\Phi_d}$, we can change $f$ by successive application of $R_2$-moves so that the monodromy of $f|_{M\setminus (f^{-1}(f(Z_1))\cup f^{-1}(Z_2))}$ along $\gamma$ corresponds to the element $\varphi$. 

\end{thm}

\section{Relation between vanishing cycles and flip and slip moves}\label{sec_mainalgorithm}

Let $f:M\rightarrow D^2$ be a purely wrinkled fibration satisfying the following conditions:

\begin{enumerate}

\item the set of singularities $\mathcal{S}_f$ of $f$ is an embedded circle in $\Int{M}$, 

\item the restriction $f|_{\mathcal{S}_f}$ is an embedding, 

\item either of the following conditions on regular fibers holds: 

\begin{itemize}

\item a regular fiber on the outside of $f(\mathcal{S}_f)$ is connected, while that on the inside of $f(\mathcal{S}_f)$ is disconnected, 

\item every regular fiber is connected and the genus of a regular fiber on the outside of $f(\mathcal{S}_f)$ is higher than that on the inside of $f(\mathcal{S}_f)$. 

\end{itemize}

\end{enumerate}

We fix a point $p_0\in \partial D^2$ and an identification $f^{-1}(p_0)\cong \Sigma_g$. 
Let $\varphi_0\in\mathcal{M}_g$ be the monodromy along $\partial D^2$ oriented counterclockwise around the center of $D^2$ with base point $p_0$. 
In this section, we will give an algorithm to obtain vanishing cycles in {\it one} higher genus regular fiber of a fibration obtained by applying flip and slip to $f$. 

We first consider the simplest case, that is, assume that $f$ has no cusps. 
We take a reference path $\gamma_0$ in $\partial D^2$ connecting $p_0$ to a point in the image of indefinite folds so that it satisfies $\Int{\gamma_0}\cap f(\mathcal{S}_f)=\emptyset$. 
This determines a vanishing cycle $c\subset \Sigma_g$ of indefinite folds. 
Then, it is easy to prove that $\varphi_0$ is contained in the group $\Ker{\Phi_c}$. 
To give an algorithm precisely, we prepare several conditions. 
The first condition is on an embedded path $\alpha\subset \Sigma_g$. 

\vspace{.8em}

\noindent
{\it Condition $C_1(c)$}: 
A path $\alpha\subset \Sigma_g$ intersects $c$ at the unique point $q\in c$ transversely. 

\vspace{.8em}

We take a path $\alpha\subset \Sigma_g$ so that $\alpha$ satisfies the condition $C_1(c)$. 
We put $\partial \alpha =\{w_1,w_2\}$. 
The second condition is on a simple closed curve $d\subset \Sigma_{g+1}$ and a diffeomorphism $j:\Sigma_g\setminus \{w_1,w_2\} \rightarrow \Sigma_{g+1}\setminus d $. 

\vspace{.8em}

\noindent
{\it Condition $C_2(c,\alpha)$}: 
the closure of $j(\Int{\alpha})$ in $\Sigma_{g+1}$ is a simple closed curve.  

\vspace{.8em}

We take a simple closed curve $d\subset \Sigma_{g+1}$ and a diffeomorphism $j:\Sigma_g\setminus \{w_1,w_2\} \rightarrow \Sigma_{g+1}\setminus d $ so that they satisfy the condition $C_2(c,\alpha)$. 
We put $\tilde{c}=j(c)$. 
The last condition is on an element $\varphi\in \MCG{(\Sigma_{g+1})} (\tilde{c},d)$. 

\vspace{.8em}
\noindent
{\it Condition $C_3(c,\alpha,d,j, \varphi_0)$}: 
$\Phi_{\tilde{c}}(\varphi)=1$ in $\MCG{(\Sigma_g)}(d)$ and $\Phi_d (\varphi) =\varphi_0^{-1}$ in $\MCG{(\Sigma_g)}(c)$. 

\vspace{.8em}

For the sake of simplicity, we will call the above conditions $C_1$, $C_2$ and $C_3$ if elements $c,\alpha, d,j$ and $\varphi_0$ are obvious. 

\begin{thm}\label{mainalgorithm}

Let $f:M\rightarrow D^2$ be a purely wrinkled fibration we took in the beginning of this section. 
We assume that $f$ has no cusps. 

\begin{enumerate}

\item[$\mathrm{(1)}$] Let $\tilde{f}$ be a fibration obtained by applying flip and slip to $f$. 
We take a point $q_0$ in the inside of $f(\mathcal{S}_{\tilde{f}})$, and reference paths $\hat{\gamma}_1, \hat{\gamma}_2, \hat{\gamma}_3$ and $\hat{\gamma}_4$ in $D^2$ connecting $q_0$ to a point on the respective fold arcs between cusps so that these paths appear in this order when we go around $q_0$ counterclockwise. 
We denote by $e_i\subset \tilde{f}^{-1}(q_0)$ a vanishing cycle determined by the path $\hat{\gamma}_i$. 
Then, there exist an identification $\tilde{f}^{-1}(q_0)\cong \Sigma_{g+1}$ and elements $\alpha, d, j$ and $\varphi$ satisfying the conditions $C_1, C_2$ and $C_3$ such that the following equality holds up to cyclic permutation: 
\[
(e_1,e_2,e_3,e_4) = (\tilde{c}, \alpha^\prime, d, \tilde{\alpha}), 
\] 
where $\tilde{c}=j(c)$, $\tilde{\alpha}$ is the closure of $j(\Int{\alpha})$ in $\Sigma_{g+1}$, and $\alpha^\prime =\varphi^{-1}(\tilde{\alpha})$. 

\item[$\mathrm{(2)}$] Let $\alpha, d, j$ and $\varphi$ be elements satisfying the conditions $C_1, C_2$ and $C_3$. 
We take simple closed curves $\tilde{c}, \tilde{\alpha}$ and $\alpha^\prime$ as in $\mathrm{(1)}$. 
Suppose that the genus of a higher genus fiber $g$ of $f$ is greater than or equal to $3$ when $(\tilde{c}, d)$ is not a bounding pair, and that both of the genera $g_1$ and $g_2$ are greater than or equal to $2$ when $(\tilde{c},d)$ is a bounding pair of genus $g_1$, where we put $g_2=g-g_1$. 
Then, there exists a fibration $\tilde{f}$ obtained by applying flip and slip to $f$ such that, for reference paths $\hat{\gamma}_1,\ldots,\hat{\gamma}_4$ as in $\mathrm{(1)}$, the corresponding vanishing cycles $e_1,\ldots,e_4$ satisfy the following equality up to cyclic permutation: 
\[
(e_1,e_2,e_3,e_4) = (\tilde{c}, \alpha^\prime, d, \tilde{\alpha}).  
\]

\end{enumerate}

\end{thm}

\begin{proof}[Proof of $\mathrm{(1)}$ of Theorem \ref{mainalgorithm}]

For a fibration $\tilde{f}:M\rightarrow D^2$, we take points $q_0,q_0^\prime, q_0^{\prime\prime}, q_1,q_1^\prime, q_1^{\prime\prime}\in D^2$ and paths $\tilde{\gamma}_0, \tilde{\gamma}_1,\tilde{\gamma}_2, \delta_0, \delta_1\subset D^2$ as in Figure \ref{afterflips_algorithm}. 

\begin{figure}[htbp]
\begin{center}
\includegraphics[width=100mm]{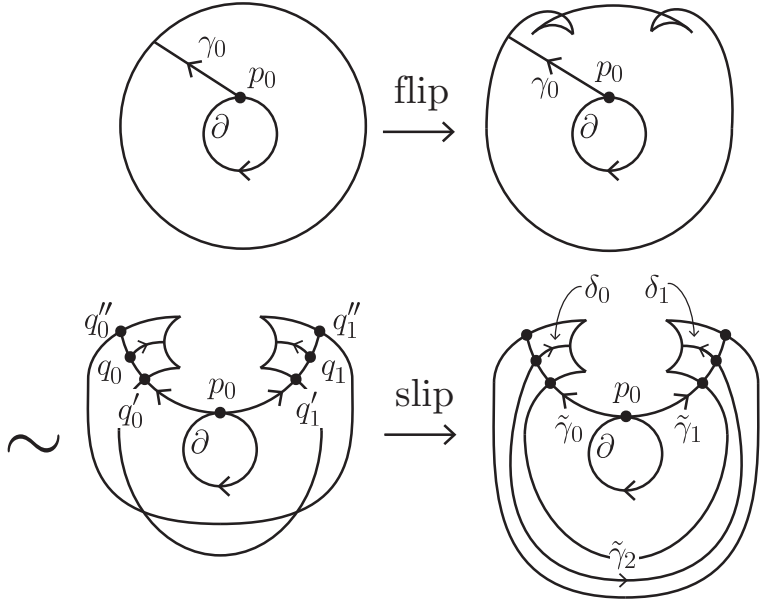}
\end{center}
\caption{the points $q_0, q_1$ is in the region with the highest genus fibers, while the points $q_0^\prime, q_0^{\prime\prime}, q_1^\prime, q_1^{\prime\prime}$ is on the image of singular locus. 
the path $\tilde{\gamma}_0$ connects $p_0$ to $q_0^{\prime\prime}$ and the path $\delta_0$ connects $q_0$ to a point in the image of singular locus. $\tilde{\gamma}_1$ and $\delta_1$ are taken similarly. The path $\tilde{\gamma}_2$ connects $q_0$ to $q_1$. 
Note that we regard $D^2$ as the subset of $S^2=\mathbb{R}^2\cup \{\infty\}$ and that $\infty\in D^2$. }
\label{afterflips_algorithm}
\end{figure}

We take an identification between the region $\Omega\subset D^2$ described in Figure \ref{regionOmega_algorithm} and the rectangle $I\times I$ so that the paths $\tilde{\gamma}_0, \tilde{\gamma}_1$ is contained in the side edges of the rectangle, the path $\tilde{\gamma}_2$ corresponds to the center horizontal line, and the image of singular loci correspond to horizontal lines (see the right side of Figure \ref{regionOmega_algorithm}). 
For each $x\in\tilde{\gamma}_2$, we denote by $u_x$ (resp. $l_x$) the vertical path which connects $x$ to the upper (resp. lower) singular image as in the right side of Figure \ref{regionOmega_algorithm}. 

\begin{figure}[htbp]
\begin{center}
\includegraphics[width=120mm]{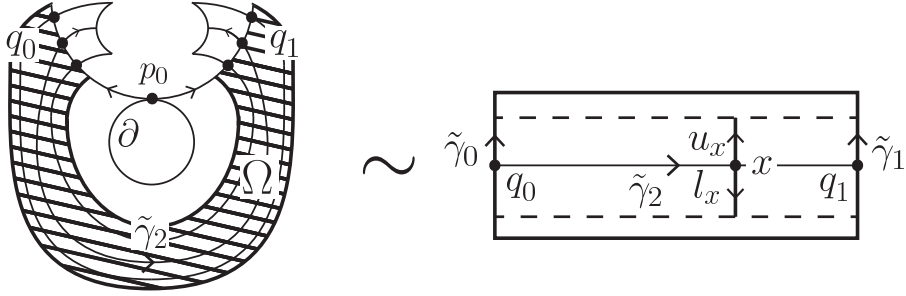}
\end{center}
\caption{the shaded region in the left figure is the region $\Omega$. 
The horizontal line with arrow in the right figure describes the path $\tilde{\gamma}_2$, while the horizontal dotted lines describe images of the singular loci. }
\label{regionOmega_algorithm}
\end{figure}

We take a horizontal distribution $\mathcal{H}$ of $\tilde{f}|_{M\setminus\mathcal{S}_{\tilde{f}}}$ so that it satisfies the following conditions: 

\begin{enumerate}

\item let $w_1^{(i)},w_2^{(i)}$ be points in $\tilde{f}^{-1}(p_0)$ which converges to an indefinite fold when $\tilde{f}^{-1}(p_0)$ approaches the singular fiber $\tilde{f}^{-1}(q_i^\prime)$ along $\tilde{\gamma}_i$ using $\mathcal{H}$. 
The set $\{w_1^{(0)},w_2^{(0)}\}$ corresponds to the set $\{w_1^{(1)},w_2^{(1)}\}$, 

\item let $d^{(i)}$ (resp. $\tilde{c}^{(i)}$) be simple closed curves in $\tilde{f}^{-1}(q_i)$ which converges to an indefinite fold when $\tilde{f}^{-1}(q_i)$ approaches the singular fiber $\tilde{f}^{-1}(q_i^{\prime})$ (resp. $\tilde{f}^{-1}(q_i^{\prime\prime})$) along $\tilde{\gamma}_i$ using $\mathcal{H}$. 
For each $i=0,1$, $d^{(i)}$ is disjoint from $\tilde{c}^{(i)}$, 

\item we obtain a diffeomorphism $j_i: \tilde{f}^{-1}(p_0)\setminus \{w_1,w_2\}\rightarrow \tilde{f}^{-1}(q_i)\setminus d^{(i)}$ by using horizontal lift of the curve $\tilde{\gamma}_i$. 
By the condition (2), $j_i^{-1}(\tilde{c}^{(i)})$ is a simple closed curve in $\tilde{f}^{-1}(p_0)$. 
$j_0^{-1}(\tilde{c}^{(0)})$ corresponds to $j_1^{-1}(\tilde{c}^{(1)})$ and these curves are equal to $c$, 

\item let $\tilde{\alpha}^{(i)}$ be a simple closed curve in $\tilde{f}^{-1}(q_i)$ which converges to an indefinite fold when $\tilde{f}^{-1}(p_0)$ approaches a singular fiber along $\delta_i$ using $\mathcal{H}$. 
$\tilde{\alpha}^{(i)}$ intersects both of the curves $\tilde{c}^{(i)}$ and $d^{(i)}$ transversely, 

\item $\sharp(\tilde{\alpha}^{(i)}\cap d^{(i)})=\sharp(\tilde{c}^{(i)}\cap \tilde{\alpha}^{(i)})=1$, 

\item by the conditions (4) and (5), the closure of $j_i^{-1}(\tilde{\alpha}^{(i)}\setminus d^{(i)})$ is a segment between $w_1^{(i)}$ and $w_2^{(i)}$. 
The closure of $j_1^{-1}(\tilde{\alpha}^{(0)}\setminus d^{(0)})$ corresponds to that of $j_2^{-1}(\tilde{\alpha}^{(1)}\setminus d^{(1)})$, 

\item since the path $\tilde{\gamma}_2$ does not contain the critical value of $\tilde{f}$, this path, together with $\mathcal{H}$, gives a diffeomorphism from $\tilde{f}^{-1}(q_0)$ to $\tilde{f}^{-1}(x)$ for each $x\in\tilde{\gamma}_2$. 
This diffeomorphism sends the curve $d^{(0)}$ (resp. $\tilde{c}^{(0)}$) to the curve $d_x$ (resp. $\tilde{c}_x$), where $d_x$ (resp. $\tilde{c}_x$) is a simple closed curve in $\tilde{f}^{-1}(x)$ which converges to an indefinite fold when $\tilde{f}^{-1}(x)$ approaches a singular fiber along $u_x$ (resp. $l_x$) using $\mathcal{H}$. 

\end{enumerate}

We choose indices of $w_1^{(i)}$ and $w_2^{(i)}$ so that $w_1^{(0)}$ corresponds to $w_1^{(1)}$. 
We put $w_i=w_i^{(0)}=w_i^{(1)}$. 
We denote by $\alpha$ the closure of $j_0^{-1}(\tilde{\alpha}^{(0)}\setminus d^{(0)})$ (which corresponds to the closure of $j_1^{-1}(\tilde{\alpha}^{(1)}\setminus d^{(1)})$).  
Since we fixed an identification $\tilde{f}^{-1}(p_0)\cong \Sigma_g$, we can regard $w_1,w_2$ as points in $\Sigma_g$. 
We can also regard $\alpha$ as a segment in $\Sigma_g$ between $w_1$ and $w_2$. 
We choose an identification $\Sigma_g\setminus \{w_1,w_2\} \cong \Sigma_{g+1}\setminus d$, where $d\subset \Sigma_{g+1}$ is a non-separating simple closed curve, so that the induced identification between $\Sigma_{g+1}\setminus d$ and $\tilde{f}^{-1}(q_i)\setminus d^{(i)}$ can be extended to an identification between $\Sigma_{g+1}$ and $\tilde{f}^{-1}(q_i)$ (to take such an identification, we modify $\mathcal{H}$ if necessary). 
By using this identification, we can regard $\tilde{c}^{(i)}$ as a curve in $\Sigma_{g+1}$, which we denote by $\tilde{c}$. 
We denote the identification between $\Sigma_{g+1}$ and $\tilde{f}^{-1}(q_i)$ as follows: 
\[
\theta_i: \Sigma_{g+1} \xrightarrow{\cong} \tilde{f}^{-1}(q_i) \hspace{.8em} (i=0,1). 
\]
On the other hand, we obtain a diffeomorphism between $\tilde{f}^{-1}(q_0)$ and $\tilde{f}^{-1}(q_1)$ by taking horizontal lifts of $\tilde{\gamma}_2$ using $\mathcal{H}$. 
We denote this diffeomorphism as follows: 
\[
\theta_3: \hat{f}^{-1}(q_0)\xrightarrow{\cong} \hat{f}^{-1}(q_1). 
\]
By the condition (7) on $\mathcal{H}$, the diffeomorphism sends $d^{(0)}$ (resp. $\tilde{c}^{(0)}$) to the curve $d^{(1)}$ (resp. $\tilde{c}^{(1)}$). 
Thus, the isotopy class $[\theta_2^{-1}\circ \theta_3\circ \theta_1]$ is contained in the subgroup $\MCG{(\Sigma_{g+1})}(\tilde{c},d)$ of the mapping class group $\mathcal{M}_{g+1}$. 
We denote this class by $\varphi\in \MCG{(\Sigma_{g+1})}(\tilde{c},d)$. 

We denote by $\tilde{\gamma}_2\cdot \delta_1$ be the path in $D^2$, starting at the point $q_0$, obtained by connecting $\tilde{\gamma}_2$ to $\delta_1$. 
This path gives the fiber $\tilde{f}^{-1}(q_0)$ a vanishing cycle of $\tilde{f}$. 
This vanishing cycle is equal to the curve $\theta_3^{-1} (\tilde{\alpha}^{(1)}) = \theta_3^{-1} \circ \theta_2 (\tilde{\alpha})$. 
This curve corresponds to the curve $\theta_1^{-1} \circ \theta_3^{-1} \circ \theta_2 (\tilde{\alpha}) = \varphi^{-1}(\tilde{\alpha}) \subset \Sigma_{g+1}$ under the identification $\theta_1$. 
Thus, the proof is completed once we prove the following lemma. 

\begin{lem}\label{lem_propertyvarphi_algorithm}

$\Phi_{\tilde{c}}(\varphi)=1$ and $\Phi_d(\varphi)=\varphi_0^{-1}$.

\end{lem}

\begin{proof}[Proof of Lemma \ref{lem_propertyvarphi_algorithm}]

The image $\Phi_{d}(\varphi)$ is equal to the monodromy along the curve $\delta_h$ described in the left side of Figure \ref{referencepath_imagePhi_algorithm}, which corresponds to $\varphi_0^{-1}$.  
Thus, we have $\Phi_{d}(\varphi)=\varphi_0^{-1}$. 
To prove $\Phi_{\tilde{c}}(\varphi)=1$, we consider the fibration obtained by applying unsink to $\tilde{f}$. 
We take the path $\tilde{\gamma}_2^\prime$ connecting $q_0$ to $q_1$ as in the right side of Figure \ref{referencepath_imagePhi_algorithm}. 
It is easy to see that the monodromy along this path corresponds to $(t_{t_d(\tilde{\alpha})}\cdot t_{t_{\tilde{\alpha}}(\tilde{c})})\cdot \varphi \cdot (t_{t_d(\tilde{\alpha})}\cdot t_{t_{\tilde{\alpha}}(\tilde{c})})^{-1}$. 
This preserves the curve $d$ and the image $\Phi_{d}((t_{t_d(\tilde{\alpha})}\cdot t_{t_{\tilde{\alpha}}(\tilde{c})})\cdot \varphi \cdot (t_{t_d(\tilde{\alpha})}\cdot t_{t_{\tilde{\alpha}}(\tilde{c})})^{-1})$ is trivial since this element is the monodromy along the curve obtained by pushing the curve $\tilde{\gamma}_2^\prime$ out of the region with the higher genus fibers, which is null-homotopic in the complement of the image of the singular loci. 
We can obtain the element $\Phi_{\tilde{c}}(\varphi)$ by taking some conjugation of $\Phi_{d}((t_{t_d(\tilde{\alpha})}\cdot t_{t_{\tilde{\alpha}}(\tilde{c})})\cdot \varphi \cdot (t_{t_d(\tilde{\alpha})}\cdot t_{t_{\tilde{\alpha}}(\tilde{c})})^{-1})$. 
In particular, $\Phi_{\tilde{c}}(\varphi)$ is also trivial and this completes the proof of Lemma \ref{lem_propertyvarphi_algorithm}.  

\end{proof}

\begin{figure}[htbp]
\begin{center}
\includegraphics[width=100mm]{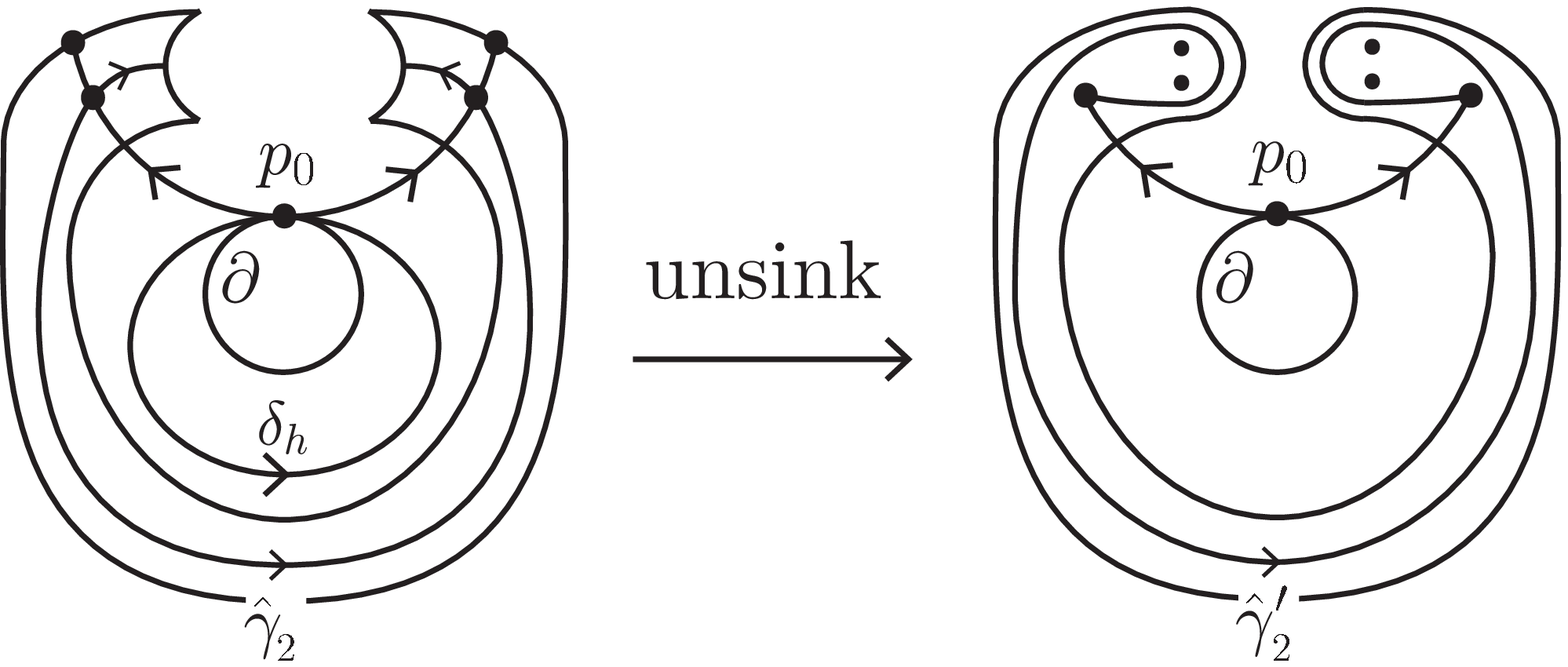}
\end{center}
\caption{}
\label{referencepath_imagePhi_algorithm}
\end{figure}

As I mentioned before the proof of Lemma \ref{lem_propertyvarphi_algorithm}, we complete the proof of (1) of Theorem \ref{mainalgorithm}. 

\end{proof}

\begin{proof}[Proof of $\mathrm{(2)}$ of Theorem \ref{mainalgorithm}]

In the proof of (1) of Theorem \ref{mainalgorithm}, we take a horizontal distribution of $\tilde{f}|_{M\setminus \mathcal{S}_{\tilde{f}}}$ and an identification $\Sigma_{g}\setminus \{w_1,w_2\}\cong \Sigma_{g+1}\setminus d$. 
Once we take these auxiliary data, we can get vanishing cycles of $\tilde{f}$ in canonical way. 
We first take a horizontal distribution of $\tilde{f}|_{M\setminus \mathcal{S}_{\tilde{f}}}$ so that an embedded path $\alpha\subset \Sigma_{g}$ determined by the distribution corresponds to the given one. 
We next take an identification $\Sigma_{g}\setminus \{w_1,w_2\}\cong \Sigma_{g+1}\setminus d$ by using the given $d,j$. 
The element $[\theta_2^{-1}\circ \theta_3\circ \theta_1]$, which appear in the proof of (1) of Theorem \ref{mainalgorithm}, is canonically determined by the chosen horizontal distribution of $\tilde{f}|_{M\setminus \mathcal{S}_{\tilde{f}}}$ of $M$ and the chosen homotopy from $f$. 

Let $\Omega$ be the region in $D^2$ as in Figure \ref{regionOmega_algorithm}. 
We take an identification $\Omega\cong I\times \tilde{\gamma}_2$. 
We also take a diffeomorphism $\Theta:\tilde{f}^{-1}(\tilde{\gamma}_0\cap \Omega)\rightarrow \tilde{f}^{-1}(\tilde{\gamma}_1\cap \Omega)$ so that it satisfies $\tilde{f}\circ \Theta= i\circ \tilde{f}$, where $i: I\times \{q_0\}\ni (t,q_0)\mapsto (t,q_1)\in I\times \{q_1\}$, and that a $4$-manifold $\tilde{f}^{-1}(\Omega)/\Theta$ is the trivial $N$-bundle over $S^1$, where $N$ is a $3$-manifold defined in Section \ref{sec_BLFoverannulus}. 
For any two elements $\varphi_1,\varphi_2\in \MCG{(\Sigma_{g+1},\tilde{c},d)}$ satisfying the condition $C_3$, the element $\varphi_1 \cdot \varphi_2^{-1}$ is contained in the group $\Ker{\Phi_{\tilde{c}}}\cap \Ker{\Phi_d}$. 
Thus, Theorem \ref{keythm_monodromyalonggamma} implies that we can change $f$ into $\tilde{f}$ by a flip and slip move so that the resulting element $[\theta_2^{-1}\circ \theta_3\circ \theta_1]$ corresponds to $\varphi^{-1}\in\MCG{(\Sigma_{g+1})}(\tilde{c},d)$ for the given $\varphi$. 
This completes the proof of (2) of Theorem \ref{mainalgorithm}. 

\end{proof}

We next consider the case that $f$ has cusps. 
We denote by $\{s_1,\ldots, s_n\}$ the set of cusps of $f$. 
We put $u_i=f(s_i)$. 
The indices of $s_i$ are chosen so that $u_1,\ldots, u_n$ appear in this order when we travel the image $f(\mathcal{S}_f)$ clockwise around a point inside $f(\mathcal{S}_f)$. 
The points $u_1,\ldots, u_n$ divides the image $f(\mathcal{S}_f)$ into $n$ edges. 
We denote by $l_i\subset f(\mathcal{S}_f)$ the edge between $u_i$ and $u_{i+1}$, where we put $u_{n+1}=u_1$. 
For a point $p_0\in\partial D^2$, we take reference paths $\gamma_1,\ldots, \gamma_n\subset D^2$ satisfying the following conditions: 

\begin{itemize}

\item $\gamma_i$ connects $p_0$ to a point in $\Int{l_i}$, 

\item $\gamma_i\cap \gamma_j=\{p_0\}$ for all $i\neq j$, 

\item $\Int{\gamma_i}\cap f(\mathcal{S}_f)=\emptyset$, 

\item $\gamma_1,\ldots, \gamma_n$ appear in this order when we go around $p_0$ counterclockwise. 

\end{itemize}

Let $\gamma_{n+1}$ be a path obtained by connecting $\partial D^2$ oriented clockwise around the center of $D^2$ to $\gamma_1$. 
The paths give $f^{-1}(p_0)\cong \Sigma_{g}$ vanishing cycles $c_1,\ldots, c_{n+1}$. 
Note that, for each $i\in \{1,\ldots, n\}$, $c_{i}$ intersects $c_{i+1}$ at a unique point transversely. 
In particular, every simple closed curve $c_i$ is non-separating. 
We also remark that $c_{n+1}$ corresponds to $\varphi_0(c_{1})$. 

Let $\hat{f}: M\rightarrow D^2$ be a fibration obtained by changing all the cusp singularities of $f$ into Lefschetz singularities by applying unsink to $f$ $n$ times.  
We take paths $\varepsilon_1,\ldots, \varepsilon_n$ in $D^2$ satisfying the following conditions: 

\begin{itemize}

\item $\varepsilon_i$ connects $p_0$ to the image of a Lefschetz singularity derived from $s_i$, 

\item $\varepsilon_i\cap \varepsilon_j=\{p_0\}$ for all $i\neq j$, 

\item $\Int{\varepsilon_i}\cap \hat{f}(\mathcal{S}_{\hat{f}})=\emptyset$, 

\item $\gamma_1,\varepsilon_1, \gamma_2, \ldots, \gamma_n, \varepsilon_n,\gamma_{n+1}$ appear in this order when we go around $p_0$ counterclockwise. 

\end{itemize}

The path $\varepsilon_i$ gives a vanishing cycle of a Lefschetz singularity of $\hat{f}$, which corresponds to the curve $t_{c_i}(c_{i+1})$. 
Let $\gamma_0$ be a based loop in $D^2\setminus \hat{f}(\mathcal{S}_{\hat{f}})$ with base point $p_0$ which is homotopic to the loop obtained by connecting $p_0$ to $\hat{f}(\mathcal{S}_{\hat{f}})$ oriented counterclockwise around a point inside $f(\mathcal{S}_f)$ using $\gamma_1$. 
It is easy to see that the monodromy along $\gamma_0$ corresponds to the following element: 
\[
\hat{\varphi}_0 = \varphi_0 \cdot (t_{t_{c_1}(c_2)} \cdot \cdots \cdot t_{t_{c_n}(c_{n+1})})^{-1}. 
\]
This element preserves the curve $c_1$ and is contained in the kernel of the homomorphism $\Phi_{c_1}$. 

Since application of flip and slip to $f$ is equivalent to application of flip and slip to $\hat{f}$ followed by application of sink $n$ times, we can obtain vanishing cycles of a fibration obtained by applying flip and slip to $f$ in the way quite similar to that in the case $f$ has no cusps. 
In order to give the precise algorithm to obtain vanishing cycles, we prepare several conditions. 

\vspace{.8em}

\noindent
{\it Condition $\tilde{C}_1(c_1,\ldots,c_n)$}: 
A path $\alpha\subset \Sigma_g$ intersects $c_1$ at the unique point $q\in c_1$ transversely.  
Furthermore, $\partial \alpha \cap (c_1\cup \cdots \cup c_{n+1})=\emptyset$. 

\vspace{.8em}

We take a path $\alpha\subset \Sigma_g$ so that $\alpha$ satisfies the condition $\tilde{C}_1(c_1,\ldots,c_n)$. 
We put $\partial \alpha =\{w_1,w_2\}$. 
The second condition is on a simple closed curve $d\subset \Sigma_{g+1}$ and a diffeomorphism $j:\Sigma_g\setminus \{w_1,w_2\} \rightarrow \Sigma_{g+1}\setminus d $. 

\vspace{.8em}

\noindent
{\it Condition $\tilde{C}_2(c_1,\ldots,c_n,\alpha)$}: 
the closure of $j(\Int{\alpha})$ in $\Sigma_{g+1}$ is a simple closed curve.  

\vspace{.8em}

We take a simple closed curve $d\subset \Sigma_{g+1}$ and a diffeomorphism $j:\Sigma_g\setminus \{w_1,w_2\} \rightarrow \Sigma_{g+1}\setminus d $ so that they satisfy the condition $\tilde{C}_2(c_1,\ldots,c_n,\alpha)$. 
We put $\tilde{c}_1=j(c_1)$. 
The third condition is on an element $\varphi\in \MCG{(\Sigma_{g+1})} (\tilde{c}_1,d)$. 

\vspace{.8em}
\noindent
{\it Condition $\tilde{C}_3(c_1,\ldots,c_n,\alpha,d,j, \varphi_0)$}: 
$\Phi_{\tilde{c}_1}(\varphi)=1$ in $\MCG(\Sigma_g)(d)$ and $\Phi_d (\varphi) =\hat{\varphi}_0^{-1}$ in $\MCG(\Sigma_g)(c_1)$. 

\vspace{.8em}
 
The last condition is on simple closed curves $\tilde{c}_2,\ldots,\tilde{c}_{n+1}\subset \Sigma_{g+1}\setminus d$. 

\vspace{.8em}
\noindent
{\it Condition $\tilde{C}_4(c_1,\ldots,c_n,\alpha,d,j)$}: 
For each $i\in\{2,\ldots,n+1\}$, $i(\tilde{c}_i)$ is isotopic to $c_i$ in $\Sigma_g$, where $i$ is an embedding defined as follows: 
\[
i: \Sigma_{g+1}\setminus d \xrightarrow{\hspace{.2em}j^{-1}\hspace{.2em}} \Sigma_g\setminus \{w_1,w_2\} \hookrightarrow \Sigma_g. 
\]
Furthermore, for each $i=1,\ldots,n$, $\tilde{c}_i$ intersects $\tilde{c}_{i+1}$ at a unique point transversely. 

\vspace{.8em} 

As the case $f$ has no cusps, we will call the above conditions $\tilde{C}_1, \tilde{C}_2, \tilde{C}_3$ and $\tilde{C}_4$ if elements $c_1,\ldots,c_n, \alpha,d,j$ and $\varphi_0$ are obvious. 
We can prove the following theorem by the argument similar to that in the proof of Theorem \ref{mainalgorithm}. 

\begin{thm}\label{mainalgorithmwithcusps}

Let $f:M\rightarrow D^2$ be a purely wrinkled fibration we took in the beginning of this section. 
Suppose that $f$ has $n>0$ cusps. 
We take vanishing cycles $c_1,\ldots, c_{n+1}$ as above. 

\begin{enumerate}

\item[$\mathrm{(1)}$] Let $\tilde{f}$ be a fibration obtained by applying flip and slip to $f$. 
We take a point $q_0$ in the inside of $f(\mathcal{S}_{\tilde{f}})$, and reference paths $\hat{\gamma}_1,\ldots, \hat{\gamma}_{n+4}$ in $D^2$ connecting $q_0$ to a point on the respective fold arcs between cusps so that these paths appear in this order when we go around $q_0$ counterclockwise. 
We denote by $e_i\subset \tilde{f}^{-1}(q_0)$ a vanishing cycle determined by the path $\hat{\gamma}_i$. 
Then, there exist an identification $\tilde{f}^{-1}(q_0)\cong \Sigma_{g+1}$ and elements $\alpha, d, j, \tilde{c}_2,\ldots, \tilde{c}_{n+1}$ and $\varphi$ satisfying the conditions $\tilde{C}_1, \tilde{C}_2, \tilde{C}_3$ and $\tilde{C}_4$ such that the following equality holds up to cyclic permutation: 
\[
(e_1,\ldots ,e_{n+4}) = (\tilde{c}_1,\ldots , \tilde{c}_{n+1}, \alpha^\prime, d, \tilde{\alpha}), 
\] 
where $\tilde{c}_1=j(c_1)$, $\tilde{\alpha}$ is the closure of $j(\Int{\alpha})$ in $\Sigma_{g+1}$, and $\alpha^\prime$ is defined as follows: 
\[
\alpha^\prime =\left(\varphi^{-1}\cdot t_{t_{\tilde{c}_1}(\tilde{c}_2)}\cdot \cdots\cdot t_{t_{\tilde{c}_n}(\tilde{c}_{n+1})}\right)(\tilde{\alpha}). 
\]

\item[$\mathrm{(2)}$] Let $\alpha, d, j, \tilde{c}_2,\ldots, \tilde{c}_{n+1}$ and $\varphi$ be elements satisfying the conditions $\tilde{C}_1,\tilde{C}_2,\tilde{C}_3$ and $\tilde{C}_4$. 
We take simple closed curves $\tilde{c}_1, \tilde{\alpha}$ and $\alpha^\prime$ as in $\mathrm{(1)}$. 
Suppose that the genus of higher genus fibers $g$ of $f$ is greater than or equal to $3$. 
Then, there exists a fibration $\tilde{f}$ obtained by applying flip and slip to $f$ such that, for reference paths $\hat{\gamma}_1,\ldots, \hat{\gamma}_{n+4}$ as in $\mathrm{(1)}$, the corresponding vanishing cycles $e_1,\ldots, e_{n+4}$ satisfy the following equality up to cyclic permutation: 
\[
(e_1,\ldots ,e_{n+4}) = (\tilde{c}_1,\ldots , \tilde{c}_{n+1}, \alpha^\prime, d, \tilde{\alpha}). 
\]

\end{enumerate}

\end{thm}

\section{Fibrations with small fiber genera}\label{sec_algorithm_smallgenera}

Although the statement (1) of Theorem \ref{mainalgorithm} holds for a fibration with an arbitrary fiber genera, the statement (2) of Theorems \ref{mainalgorithm} and \ref{mainalgorithmwithcusps} do not hold if genera of fibers are too small. 
The main reason of this is non-triviality of the group $\pi_1(\Diff^+{\Sigma_{g-1}}, \text{id})$ when $g<3$. 
To deal with fibrations with small fiber genera, we need to look at additional data on sections of fibrations. 
Let $f:M\rightarrow D^2$ be a purely wrinkled fibration we took in the beginning of Section \ref{sec_mainalgorithm}. 

\subsection{Case 1: every fiber of $f$ is connected}

In this subsection, we assume that every fiber of $f$ is connected. 
We first consider the case $f$ has no cusps. 
We take a point $p_0$, an identification $f^{-1}(p_0)\cong \Sigma_g$, a reference path $\gamma_0\subset D^2$, a vanishing cycle $c\subset \Sigma_g$, and a monodromy $\varphi_0\in \MCG{(\Sigma_g)}(c)$ as we took in Section \ref{sec_mainalgorithm}. 
It is easy to see that $f$ has a section. 
We take a section $\sigma: D^2\rightarrow M$ of $f$. 
We put $x=\sigma(p_0)$, which is contained in the complement $\Sigma_g\setminus c$. 
This section gives a lift $\tilde{\varphi}\in \MCG{(\Sigma_g,x)}(c)$. 
It is easy to show that this element is contained in the kernel of the following homomorphism: 
\[
\Phi_c^{x}: \MCG{(\Sigma_g,x)}(c)\rightarrow \MCG{(\Sigma_{g-1},x)}, 
\]
which is defined as we define $\Phi_c$. 

As in Section \ref{sec_mainalgorithm}, we give several conditions. 
The first condition is on an embedded path $\alpha\subset \Sigma_g\setminus \{x\}$. 

\vspace{.8em}

\noindent
{\it Condition $C_1^\prime(c,\sigma)$}: 
A path $\alpha\subset \Sigma_g\setminus \{x\}$ intersects $c$ at the unique point $q\in c$ transversely. 

\vspace{.8em}

We take a path $\alpha\subset \Sigma_g\setminus \{x\}$ so that $\alpha$ satisfies the condition $C_1^\prime(c,\sigma)$. 
We put $\partial \alpha =\{w_1,w_2\}$. 
The second condition is on a simple closed curve $d\subset \Sigma_{g+1}$ and a diffeomorphism $j:\Sigma_g\setminus \{w_1,w_2\} \rightarrow \Sigma_{g+1}\setminus d $. 

\vspace{.8em}

\noindent
{\it Condition $C_2^\prime(c,\alpha,\sigma)$}: 
the closure of $j(\Int{\alpha})$ in $\Sigma_{g+1}$ is a simple closed curve.  

\vspace{.8em}

We take a simple closed curve $d\subset \Sigma_{g+1}$ and a diffeomorphism $j:\Sigma_g\setminus \{w_1,w_2\} \rightarrow \Sigma_{g+1}\setminus d $ so that they satisfy the condition $C_2^\prime(c,\alpha,\sigma)$. 
We put $\tilde{c}=j(c)$ and $\tilde{x}=j(x)$. 
The last condition is on an element $\varphi\in \MCG{(\Sigma_{g+1},\tilde{x})} (\tilde{c},d)$. 

\vspace{.8em}
\noindent
{\it Condition $C_3^\prime(c,\alpha,d,j, \varphi_0,\sigma)$}: 
$\Phi_{\tilde{c}}^{\tilde{x}}(\varphi)=1$ in $\MCG{(\Sigma_g,\tilde{x})}(d)$ and $\Phi_d^{\tilde{x}} (\varphi) =\tilde{\varphi}_0^{-1}$ in $\MCG{(\Sigma_g,x)}(c)$. 

\begin{thm}\label{mainalgorithmwithsection}

Let $f:M\rightarrow D^2$ be a purely wrinkled fibration as above. 

\begin{enumerate}

\item[$\mathrm{(1)}$] Let $\tilde{f}$ be a fibration obtained by applying flip and slip to $f$. 
We take a point $q_0$, reference paths $\hat{\gamma}_1,\ldots, \hat{\gamma}_4$ in $D^2$ and $e_i\subset \tilde{f}^{-1}(q_0)$ as in $\mathrm{(1)}$ of Theorem \ref{mainalgorithm}. 
Then, there exist an identification $\tilde{f}^{-1}(q_0)\cong \Sigma_{g+1}$ and elements $\alpha, d, j$ and $\varphi$ satisfying the conditions $C_1^\prime, C_2^\prime$ and $C_3^\prime$ such that the following equality holds up to cyclic permutation: 
\[
(e_1,e_2,e_3,e_4) = (\tilde{c}, \alpha^\prime, d, \tilde{\alpha}), 
\] 
where $\tilde{c}=j(c)$, $\tilde{\alpha}$ is the closure of $j(\Int{\alpha})$ in $\Sigma_{g+1}$, and $\alpha^\prime =\varphi^{-1}(\tilde{\alpha})$. 

\item[$\mathrm{(2)}$] Let $\alpha, d, j$ and $\varphi$ be elements satisfying the conditions $C_1^\prime, C_2^\prime$ and $C_3^\prime$. 
We take simple closed curves $\tilde{c}, \tilde{\alpha}$ and $\alpha^\prime$ as in $\mathrm{(1)}$. 
Suppose that the genus $g$ is greater than or equal to $2$. 
Then, there exists a fibration $\tilde{f}$ obtained by applying flip and slip to $f$ such that, for reference paths $\hat{\gamma}_1,\ldots,\hat{\gamma}_4$ as in $\mathrm{(1)}$, the corresponding vanishing cycles $e_1,\ldots, e_4$ satisfy the following equality up to cyclic permutation: 
\[
(e_1,e_2,e_3,e_4) = (\tilde{c}, \alpha^\prime, d, \tilde{\alpha}).  
\]

\end{enumerate}

\end{thm}

\begin{proof}[Proof of $\mathrm{(1)}$ of Theorem \ref{mainalgorithmwithsection}]

The proof of (1) of Theorem \ref{mainalgorithmwithsection} is quite similar to that of (1) of Theorem \ref{mainalgorithm}. 
The only difference is the following point: 
instead of a horizontal distribution $\mathcal{H}$ of the fibration $\tilde{f}|_{M\setminus \mathcal{S}_{\tilde{f}}}$, we take a horizontal distribution $\mathcal{H}_{\sigma}$ of the fibration $\tilde{f}|_{M\setminus \mathcal{S}_{\tilde{f}}}$, which satisfies the same conditions as that on $\mathcal{H}$, so that it is tangent to the image of the section $\sigma$. 
By using such a horizontal distribution, we can apply all the arguments in the proof of Theorem \ref{mainalgorithm} straightforwardly. 
We omit details of the proof. 

\end{proof}

\begin{proof}[Proof of $\mathrm{(2)}$ of Theorem \ref{mainalgorithmwithsection}]

As the proof of (1), the proof of (2) is also similar to that of (2) of Theorem \ref{mainalgorithm}. 
By the same argument as in the proof of (2) of Theorem \ref{mainalgorithm}, all we have to prove is that we can take a homotopy from $f$ to $\tilde{f}$ so that the element $[\theta_2^{-1}\circ \theta_3\circ \theta_1]$ corresponds to $\varphi^{-1}$ for given $\varphi$. 

It is known that the group $\pi_1(\Diff^+(\Sigma_{g-1},x), \text{id})$ is trivial if $g$ is greater than or equal to $2$ (cf. \cite{Earle_Schatz}). 
Thus, by the argument similar to that in Section \ref{sec_BLFoverannulus}, we can prove that the group $\Ker{\Phi_{\tilde{c}}^{\tilde{x}}}\cap \Ker{\Phi_d^{\tilde{x}}}$ is generated by the following set: 
\[
\{t_{\tilde{\delta}(\eta)}\cdot t_{\tilde{c}}^{-1}\cdot t_{d}^{-1}\in \MCG{(\Sigma_{g+1},\tilde{x})}(\tilde{c},d) \hspace{.3em} | \hspace{.3em} \eta\in\Pi(\Sigma_{g-1} \setminus \{\tilde{x},v_i,w_j\}, v_k,w_l), \{i,k\}=\{j,l\}=\{1,2\}\},  
\]
where $\Pi(\Sigma_{g-1}\setminus \{\tilde{x},v_i,w_j\}, v_k,w_l)$ and $\tilde{\delta}(\eta)$ are defined as in Section \ref{sec_BLFoverannulus}. 
Thus, by the similar argument to that in the proof of Theorem \ref{keythm_monodromyalonggamma}, we can change $[\theta_2^{-1}\circ \theta_3\circ \theta_1]$ into $[\theta_2^{-1}\circ \theta_3\circ \theta_1]\cdot \psi$ for any $\psi\in \Ker{\Phi_{\tilde{c}}^{\tilde{x}}}\cap \Ker{\Phi_d^{\tilde{x}}}$ by modifying a flip and slip from $f$ to $\tilde{f}$. 
This completes the proof of the statement (2). 

\end{proof}

We can deal with a fibration with cusps similarly by using sink and unsink as in Section \ref{sec_mainalgorithm}. 
Suppose that $f$ has $n>0$ cusps and we take vanishing cycles $c_1,\ldots ,c_{n+1}\subset \Sigma_g$ as we took in Section \ref{sec_mainalgorithm}. 
We also take a section $\sigma:D^2\rightarrow M$ of $f$. 
We put $x=\sigma(p_0)$, which is contained in the complement $\Sigma_g\setminus (c_1\cup \cdots \cup c_{n+1})$. 
This gives a lift $\tilde{\varphi}_0\in\MCG{(\Sigma_g,x)}(c_1)$ of $\varphi_0$. 
As in Section \ref{sec_mainalgorithm}, we put $\hat{\varphi}_0=\tilde{\varphi}_0\cdot (t_{t_{\tilde{c}_1}(\tilde{c}_2)}\cdot \cdots \cdot t_{t_{\tilde{c}_n}(\tilde{c}_{n+1})})^{-1}$, and we give several conditions on elements $\alpha, d, j, \varphi, \tilde{c}_2,\ldots, \tilde{c}_{n+1}$. 
 
\vspace{.8em}

\noindent
{\it Condition $\tilde{C}_1^\prime(c_1,\ldots,c_n, \sigma)$}: 
A path $\alpha\subset \Sigma_g\setminus \{x\}$ intersects $c_1$ at the unique point $q\in c_1$ transversely.  
Furthermore, $\partial \alpha \cap (c_1\cup \cdots \cup c_{n+1})=\emptyset$. 

\vspace{.8em}

\noindent
{\it Condition $\tilde{C}_2^\prime(c_1,\ldots,c_n,\alpha,\sigma)$}: 
the closure of $j(\Int{\alpha})$ in $\Sigma_{g+1}$ is a simple closed curve.  

\vspace{.8em}

\noindent
{\it Condition $\tilde{C}_3^\prime(c_1,\ldots,c_n,\alpha,d,j, \varphi_0, \sigma)$}: 
$\Phi_{\tilde{c}_1}^{\tilde{x}}(\varphi)=1$ in $MCG(\Sigma_g,\tilde{x})(d)$ and $\Phi_d^{\tilde{x}}(\varphi) =\hat{\varphi}_0^{-1}$ in $MCG(\Sigma_g, x)(c_1)$, where we put $\tilde{x}=j(x)$ and $\tilde{c}=j(c)$. 

\vspace{.8em}
 
\noindent
{\it Condition $\tilde{C}_4^\prime(c_1,\ldots,c_n,\alpha,d,j,\sigma)$}: 
For each $i\in\{2,\ldots,n+1\}$, $i(\tilde{c}_i)$ is isotopic to $c_i$ in $\Sigma_g\setminus \{x\}$, where $i$ is an embedding defined as follows: 
\[
i: \Sigma_{g+1}\setminus d \xrightarrow{\hspace{.2em}j^{-1}\hspace{.2em}} \Sigma_g\setminus \{w_1,w_2\} \hookrightarrow \Sigma_g. 
\]
Furthermore, for each $i=1,\ldots,n$, $\tilde{c}_i$ intersects $\tilde{c}_{i+1}$ at a unique point transversely. 

\vspace{.8em}

The following theorem can be proved in the way quite similar to that of the proof of Theorem \ref{mainalgorithmwithsection}.

\begin{thm}\label{mainalgorithmwithcuspswithsection}

Let $f:M\rightarrow D^2$ be a purely wrinkled fibration as above. 

\begin{enumerate}

\item[$\mathrm{(1)}$] Let $\tilde{f}$ be a fibration obtained by applying flip and slip to $f$. 
We take a point $q_0$, reference paths $\hat{\gamma}_1,\ldots, \hat{\gamma}_{n+4}$ in $D^2$, vanishing cycles $e_1,\ldots, e_{n+4}\subset \tilde{f}^{-1}(q_0)$ as we took in $\mathrm{(1)}$ of Theorem \ref{mainalgorithmwithcusps}. 
Then, there exist an identification $\tilde{f}^{-1}(q_0)\cong \Sigma_{g+1}$ and elements $\alpha, d, j, \tilde{c}_2,\ldots, \tilde{c}_{n+1}$ and $\varphi$ satisfying the conditions $\tilde{C}_1^\prime, \tilde{C}_2^\prime, \tilde{C}_3^\prime$ and $\tilde{C}_4^\prime$ such that the following equality holds up to cyclic permutation: 
\[
(e_1,\ldots ,e_{n+4}) = (\tilde{c}_1,\ldots , \tilde{c}_{n+1}, \alpha^\prime, d, \tilde{\alpha}), 
\] 
where $\tilde{c}_1=j(c_1)$, $\tilde{\alpha}$ is the closure of $j(\Int{\alpha})$ in $\Sigma_{g+1}$, and $\alpha^\prime$ is defined as follows: 
\[
\alpha^\prime =\left(\varphi^{-1}\cdot t_{t_{\tilde{c}_1}(\tilde{c}_2)}\cdot \cdots\cdot t_{t_{\tilde{c}_n}(\tilde{c}_{n+1})}\right)(\tilde{\alpha}). 
\]

\item[$\mathrm{(2)}$] Let $\alpha, d, j, \tilde{c}_2,\ldots, \tilde{c}_{n+1}$ and $\varphi$ be elements satisfying the conditions $\tilde{C}_1^\prime,\tilde{C}_2^\prime,\tilde{C}_3^\prime$ and $\tilde{C}_4^\prime$. 
We take simple closed curves $\tilde{c}_1, \tilde{\alpha}$ and $\alpha^\prime$ as in $\mathrm{(1)}$. 
Suppose that the genus $g$ is greater than or equal to $2$. 
Then, there exists a fibration $\tilde{f}$ obtained by applying flip and slip move to $f$ such that, for a reference path $\hat{\gamma}_1,\ldots, \hat{\gamma}_{n+4}$ as in $\mathrm{(1)}$, the corresponding vanishing cycles $e_1,\ldots, e_{n+4}$ satisfy the following equality up to cyclic permutation: 
\[
(e_1,\ldots ,e_{n+4}) = (\tilde{c}_1,\ldots , \tilde{c}_{n+1}, \alpha^\prime, d, \tilde{\alpha}). 
\]

\end{enumerate}

\end{thm}

\subsection{Case 2: $f$ has disconnected fibers}

We next consider the case $f$ has disconnected fibers. 
In this case, $f$ has no cusps. 
We take a point $p_0\in\partial D^2$, an identification $f^{-1}(p_0)\cong \Sigma_g$, a reference path $\gamma_0$, a vanishing cycle $c\subset \Sigma_g$, and a monodromy $\varphi_0\in\MCG{(\Sigma_g)}(c)$ as we took in Section \ref{sec_mainalgorithm}. 
We also take a disconnected fiber of $f$ and denote this by $S_1\amalg S_2$, where $S_i$ is a connected component of the fiber. 
We take a section $\sigma_i: D^2\rightarrow M$ of $f$ which intersects $S_i$ for each $i=1,2$. 
We put $x_i= \sigma_i(p_0)$, which is contained in the complement $\Sigma_g\setminus c$. 
The sections $\sigma_1$ and $\sigma_2$ gives a lift $\hat{\varphi}_0\in \MCG{(\Sigma_g, x_1,x_2)}(c^\text{ori})$, and this element is contained in the kernel of the following homomorphism: 
\[
\Phi_{c}^{x_1,x_2}: \MCG{(\Sigma_g, x_1,x_2)}(c^\text{ori}) \rightarrow \MCG{(\Sigma_{g_1},x_1)}\times \MCG{(\Sigma_{g_2},x_2)}, 
\]
where $g_i$ is the genus of the closed surface $S_i$. 

By using this lift, we can apply all the argument in Case 1 straightforwardly, and we can obtain the theorem similar to Theorem \ref{mainalgorithmwithsection} (we need the assumption $g_1, g_2 \geq 1$). 
We omit the details of arguments.

\begin{rem}\label{rem_casegenus1}

The statement (2) of Theorem \ref{mainalgorithmwithsection} and Theorem \ref{mainalgorithmwithcuspswithsection} does not hold if $g=1$ since the group $\pi_1(\Diff^+{(S^2,x)},\text{id})$ is not trivial (cf. \cite{Earle_Schatz}). 
To apply the same argument as in the proof of (2) of Theorem \ref{mainalgorithmwithsection} to the case $g=1$, we need to take three disjoint sections of $f$. 
Since the group $\pi_1(\Diff^+{(S^2,x_1,x_2,x_3)},\text{id})$ is trivial, the statement similar to that in Theorem \ref{mainalgorithmwithsection} and Theorem \ref{mainalgorithmwithcuspswithsection} hold for a fibration $f$ with $g=1$ (note that the group $\pi_1(\Diff^+{(S^2,x_1,x_2)},\text{id})$ is non-trivial). 
Furthermore, we can deal with a fibration with disconnected fibers which contain spheres as connected components by taking three disjoint sections so that these sections go through the sphere components. 
We omit, however, details of arguments about this case for simplicity of the paper. 

\end{rem}

\section{Application: Examples of Williams diagrams}\label{sec_exampleWilliamsdiagram}

Williams \cite{Wil2} defined a certain cyclically ordered sequence of non-separating simple closed curves in a closed surface which describes a $4$-manifold. 
This sequence is obtained by looking at vanishing cycles of  a {\it simplified purely wrinkled fibration}, which is defined below. 
In this section, we will look at relation between flip and slip and sequences of simple closed curves Williams defined. 
We will then give some new examples of this sequence. 

\begin{defn}

A purely wrinkled fibration $\zeta:M^4\rightarrow S^2$ is called a {\it simplified purely wrinkled fibration} if it satisfies the following conditions: 

\begin{enumerate}

\item all the fiber of $\xi$ are connected, 

\item the set of singularities $\mathcal{S}_{\zeta}\subset M$ of $\zeta$ is connected and non-empty, 

\item the restriction $\zeta|_{\mathcal{S}_{\zeta}}$ is injective. 

\end{enumerate}

It is easy to see that $\zeta$ has two types of regular fibers: $\Sigma_g$ and $\Sigma_{g-1}$ for some $g\geq 1$. 
We call the genus $g$ of a higher-genus regular fiber the {\it genus} of $\zeta$. 
In this paper, we call a simplified purely wrinkled fibration an SPWF for simplicity. 

\end{defn}

Let $\zeta:M\rightarrow S^2$ be a genus-$g$ SPWF. 
We denote by $\{s_1,\ldots ,s_n\}$ the set of cusps of $f$. 
We put $u_i=f(s_i)$. 
We take a regular value $p_0$ of $\zeta$ so that the genus of the fiber $\zeta^{-1}(p_0)$ is equal to $g$. 
The indices of $s_i$ are chosen so that $u_1,\ldots, u_n$ appear in this order when we travel the image $\zeta(\mathcal{S}_f)$ counterclockwise around $p_0$. 
The points $u_1,\ldots,u_n$ divides the image $\zeta(\mathcal{S}_{\zeta})$ into $n$ edges. 
We denote by $l_i\subset \zeta(\mathcal{S}_{\zeta})$ the edge between $u_i$ and $u_{i+1}$ (we regard the indices as in $\mathbb{Z}/n\mathbb{Z}$.  In particular, $u_{n+1}=u_1$). 
We take paths $\gamma_1,\ldots,\gamma_n\subset S^2$ satisfying the following conditions:

\begin{itemize}

\item $\gamma_i$ connects $p_0$ to a point in $\Int{l_i}$, 

\item $\Int{\gamma_i} \cap f(\mathcal{S}_{\zeta})=\emptyset$, 

\item $\gamma_i\cap \gamma_j=\{p_0\}$ if $i\neq j$. 

\end{itemize}

We fix an identification $\zeta^{-1}(p_0)\cong \Sigma_g$. 
These paths give $\Sigma_g$ a sequence of vanishing cycles of $\zeta$, which we denote by $(c_1,\ldots,c_n)$. 

\begin{defn}[\cite{Wil2}]

Let $\zeta:M\rightarrow S^2$ be an SPWF with genus $g\geq 3$. 
We denote by $(c_1,\ldots,c_n)$ a sequence of simple closed curves in $\Sigma_g$ obtained as above. 
We call this sequence a {\it Williams diagram} of a $4$-manifold $M$. 

\end{defn}

\begin{rem}

A diagram defined above was called a "surface diagram" of  a $4$-manifold $M$ in \cite{Wil2}. 
We can define a Williams diagram of an SPWF in the obvious way. 
In this paper, we call both of the diagram, that of a $4$-manifold and that of an SPWF, a Williams diagram. 

\end{rem}

\begin{rem}

It is known that every smooth map $h:M^4\rightarrow S^2$ from an oriented, closed, connected $4$-manifold $M$ is homotopic to an SPWF with genus greater than $2$ (see \cite{Wil}). 
In particular, every closed oriented connected $4$-manifold has a Williams diagram. 
Moreover, the total space of an SPWF is uniquely determined by a sequence of vanishing cycles if the genus is greater than $2$ since the group $\pi_1(\Diff^+{\Sigma_{g-1}}, \text{id})$ is trivial if $g\geq 3$. 
Thus, a $4$-manifold is uniquely determined by a Williams diagram. 
However, it is known that there exist infinitely many SPWFs which have same vanishing cycles (see \cite{BK} and \cite{H}, for example). 

\end{rem}


Let $\zeta:M\rightarrow S^2$ be a genus-$g$ SPWF and $(c_1,\ldots,c_n)$ a Williams diagram of $\zeta$. 
For a base point $p_0$, we take a disk $D$ in $S^2\setminus \zeta(\mathcal{S}_{\zeta})$ satisfying the following conditions: 

\begin{itemize}

\item $p_0\in \partial D$, 

\item $\gamma_i\cap D=\{p_0\}$, where $\gamma_i\subset S^2$ is a reference path from $p_0$ which gives a vanishing cycle $c_i$,  

\item $\gamma_1,\ldots, \gamma_n, D$ appear in this order when we go around $p_0$ counterclockwise. 

\end{itemize}

\noindent
We consider the restriction $\zeta|_{M\setminus \zeta^{-1}(\Int{D})}$. 
This is a purely wrinkled fibration and satisfies the conditions in the beginning of Section \ref{sec_mainalgorithm}. 
Thus, we can apply arguments in Section \ref{sec_mainalgorithm} to $\zeta|_{M\setminus \zeta^{-1}(\Int{D})}$. 
In particular, we can describe an algorithm to obtain a Williams diagram of a fibration obtained by applying flip and slip to $\zeta$. 
As in Section \ref{sec_mainalgorithm}, we prepare several conditions to give an algorithm precisely. 
We first remark that we can assume that $\varphi_0$ is trivial in this case since $\zeta^{-1}(\partial D)$ is bounded by the trivial fibration. 
In particular, we obtain: 
\[
\hat{\varphi}= (t_{t_{c_{1}}(c_{2})}\cdot \cdots \cdot t_{t_{c_{n-1}}(c_{n})}\cdot t_{t_{c_{n}}(c_{1})})^{-1}. 
\]

The first condition is on an embedded path $\alpha\subset \Sigma_g$. 

\vspace{.8em}

\noindent
{\it Condition $W_1(c_1,\ldots,c_n)$}: 
A path $\alpha\subset \Sigma_g$ intersects $c_1$ at the unique point $q\in c_1$ transversely.  
Furthermore, $\partial \alpha \cap (c_1\cup \cdots \cup c_n)=\emptyset$. 

\vspace{.8em}

We take a path $\alpha\subset \Sigma_g$ so that $\alpha$ satisfies the condition $W_1(c_1,\ldots,c_n)$. 
We put $\partial \alpha =\{w_1,w_2\}$. 
The second condition is on a simple closed curve $d\subset \Sigma_{g+1}$ and a diffeomorphism $j:\Sigma_g\setminus \{w_1,w_2\} \rightarrow \Sigma_{g+1}\setminus d $. 

\vspace{.8em}

\noindent
{\it Condition $W_2(c_1,\ldots,c_n,\alpha)$}: 
the closure of $j(\Int{\alpha})$ in $\Sigma_{g+1}$ is a simple closed curve.  

\vspace{.8em}

We take a simple closed curve $d\subset \Sigma_{g+1}$ and a diffeomorphism $j:\Sigma_g\setminus \{w_1,w_2\} \rightarrow \Sigma_{g+1}\setminus d $ so that they satisfy the condition $W_2(c_1,\ldots,c_n,\alpha)$. 
We put $\tilde{c}_1=j(c_1)$. 
The third condition is on an element $\varphi\in \MCG{(\Sigma_{g+1})} (\tilde{c}_1,d)$. 

\vspace{.8em}
\noindent
{\it Condition $W_3(c_1,\ldots,c_n,\alpha,d,j)$}: 
$\Phi_{\tilde{c}_1}(\varphi)=1$ in $\MCG(\Sigma_g)(d)$ and $\Phi_d (\varphi) =t_{t_{c_1}(c_2)}\cdot \cdots \cdot t_{t_{c_{n-1}}(c_n)}\cdot t_{t_{c_{n}}(c_1)}$ in $\MCG(\Sigma_g)(c_1)$. 

\vspace{.8em}

The last condition is on simple closed curves $\tilde{c}_2,\ldots,\tilde{c}_n\subset \Sigma_{g+1}\setminus d$. 

\vspace{.8em}
\noindent
{\it Condition $W_4(c_1,\ldots,c_n,\alpha,d,j)$}: 
For each $i\in\{2,\ldots,n\}$, $i(\tilde{c}_i)$ is isotopic to $c_i$ in $\Sigma_g$, where $i$ is an embedding defined as follows: 
\[
i: \Sigma_{g+1}\setminus d \xrightarrow{\hspace{.2em}j^{-1}\hspace{.2em}} \Sigma_g\setminus \{w_1,w_2\} \hookrightarrow \Sigma_g. 
\]
Furthermore, $\tilde{c}_i$ intersects $\tilde{c}_{i+1}$ at a unique point transversely for each $i\in \mathbb{Z}/n\mathbb{Z}$. 

\vspace{.8em}

By Theorem \ref{mainalgorithmwithcusps}, we immediately obtain the following theorem. 

\begin{thm}\label{mainalgorithm_SPWF}

Let $\zeta:M\rightarrow S^2$ be a genus-$g$ SPWF and $(c_1,\ldots,c_n)$ a Williams diagram of $\zeta$. 

\begin{enumerate}

\item[$\mathrm{(1)}$] Let $\tilde{\zeta}$ be a genus-$(g+1)$ SPWF obtained by applying flip and slip to $\zeta$. 
Then, there exist elements $\alpha, d, j, \tilde{c}_2,\ldots,\tilde{c}_n, \varphi$ satisfying the conditions $W_1, W_2, W_3$ and $W_4$ such that 
the sequence $(\tilde{c}_1,\ldots,\tilde{c}_n,\tilde{c}_1, \alpha^\prime, d, \tilde{\alpha})$ gives a Williams diagram of $\tilde{\zeta}$, where $\tilde{c}_1=j(c_1)$, $\tilde{\alpha}$ is the closure of $j(\Int{\alpha})$ in $\Sigma_{g+1}$, and $\alpha^\prime$ is defined as follows: 
\[
\alpha^\prime =\left(\varphi^{-1}\cdot t_{t_{\tilde{c}_1}(\tilde{c}_2)}\cdot\cdots\cdot t_{t_{\tilde{c}_n}(\tilde{c}_1)}\right) (\tilde{\alpha}). 
\]

\item[$\mathrm{(2)}$] Let $\alpha, d, j, \tilde{c}_2,\ldots,\tilde{c}_n$ and $\varphi$ be elements satisfying the conditions $W_1, W_2, W_3$ and $W_4$. 
Suppose that $g$ is greater than or equal to $3$. 
We take simple closed curves $\alpha^\prime,\tilde{\alpha}$ as in $\mathrm{(1)}$. 
Then, there exists a genus-$(g+1)$ SPWF $\tilde{\zeta}$ obtained by applying flip and slip to $\zeta$ such that $(\tilde{c}_1,\ldots,\tilde{c}_n, \tilde{c}_1, \alpha^\prime, d, \tilde{\alpha})$ is a Williams diagram of $\tilde{\zeta}$. 

\end{enumerate}

\end{thm}

As in Section \ref{sec_algorithm_smallgenera}, we can deal with SPWFs with small genera by looking at additional data. 
Let $\zeta:M\rightarrow S^2$ be a genus-$g$ SPWF with Williams diagram $(c_1,\ldots,c_n)$. 
We take a disk $D\subset S^2$ as above.
We also take a section $\sigma: S^2\setminus \Int{D}\rightarrow M\setminus \zeta^{-1}(\Int{D})$ of the fibration $\zeta|_{M\setminus \zeta^{-1}(\Int{D})}$. 
We put $x=\sigma(p_0)$. 
We take a trivialization $\zeta^{-1}(D)\cong D \times \Sigma_g$ so that it is compatible with the identification $\zeta^{-1}(p_0)\cong \Sigma_g$. 
Let $\beta_x\in \pi_1(\Sigma_g, x)$ be an element which is represented by the following loop: 
\[
p_2\circ \sigma: (\partial D, p_0)\rightarrow (\Sigma_g\setminus (c_1\cup \cdots \cup c_n), x), 
\]
where $p_2: D\times \Sigma_g\rightarrow \Sigma_g$ is the projection onto the second component. 
It is easy to see that the monodromy along $\partial D$ (oriented as a boundary of $S^2\setminus \Int{D}$) corresponds to the pushing map $Push(\beta_x)^{-1}$. 
Thus, we can assume that $\tilde{\varphi}_0=Push(\beta_x)^{-1}\in \MCG{(\Sigma_g,x)}(c_1)$ in this case. 
We call the loop $\beta_x$ an {\it attaching loop}. 

\begin{rem}

We can obtain a handle decomposition of the total space of an SPWF by changing it into a simplified broken Lefschetz fibration using unsink. 
Indeed, Baykur \cite{Ba} gave a way to obtain a handle decomposition of the total spaces of simplified broken Lefschetz fibrations from monodromy representation (or equivalently, vanishing cycles of the fibrations). 
The loop $t\mapsto (t, \beta_x(t)) \in D\times \Sigma_g$ corresponds to the attaching circle of the $2$-handle in the lower side of the fibration. 
This is because $\beta_x$ is called an attaching loop. 

\end{rem}

We consider the following conditions on elements $\alpha, d, j, \varphi, \tilde{c}_2,\ldots, \tilde{c}_n$ as in Section \ref{sec_algorithm_smallgenera}. 

\noindent
{\it Condition $W_1^{\prime}(c_1,\ldots,c_n,\sigma)$}: 
A path $\alpha\subset \Sigma_g\setminus \{x\}$ intersects $c_1$ at the unique point $q\in c$ transversely. 
Furthermore, $\partial \alpha \cap (c_1\cup \cdots \cup c_n)=\emptyset$. 

\vspace{.8em}
\noindent
{\it Condition $W_2^{\prime}(c_1,\ldots,c_n,\alpha,\sigma)$}: 
the closure of $j(\Int{\alpha})$ in $\Sigma_{g+1}$ is a simple closed curve.  

\vspace{.8em}
\noindent
{\it Condition $W_3^{\prime}(c_1,\ldots,c_n,\alpha,d,j,\sigma)$}: 
We put $\tilde{c}_1=j^{-1}(c_1)$ and $\tilde{x}= j(x)$. 
$\Phi_{\tilde{c}_1}^{\tilde{x}}(\varphi)=1$ in $\MCG{(\Sigma_g,x)}(d)$ and $\Phi_d^{\tilde{x}} (\varphi) =t_{t_{c_1}(c_2)}\cdot \cdots \cdot t_{t_{c_{n-1}}(c_n)} \cdot t_{t_{c_{n}}(c_1)} \cdot Push(\beta_x)$ in $\MCG(\Sigma_g,x)(c)$.

\vspace{.8em}
\noindent
{\it Condition $W_4^{\prime}(c_1,\ldots,c_n,\alpha,d,j,\sigma)$}: 
For each $i\in\{2,\ldots,n\}$, a curve $\tilde{c}_i\subset \Sigma_{g+1}\setminus \{\tilde{x}\}$ satisfies $i(\tilde{c}_i)$ is isotopic to $c_i$ in $\Sigma_g\setminus \{x\}$, where $i$ is an embedding defined as follows: 
\[
i: \Sigma_{g+1}\setminus d \xrightarrow{\hspace{.2em}j^{-1}\hspace{.2em}} \Sigma_g\setminus \partial \alpha \hookrightarrow \Sigma_g. 
\]

\vspace{.8em}

Then, we can obtain the following theorem by Theorem \ref{mainalgorithmwithcuspswithsection}. 

\begin{thm}\label{mainalgorithm_SPWFwithsection}

Let $\zeta:M\rightarrow S^2$ be a genus-$g$ SPWF and $(c_1,\ldots,c_n)$ a Williams diagram of $\zeta$. 
We take a disk $D\subset S^2$, $\sigma: S^2\setminus \Int{D}\rightarrow M\setminus \zeta^{-1}(\Int{D})$, and an element $\beta_x\in \pi_1(\Sigma_g, x)$ as above. 

\begin{enumerate}

\item[$\mathrm{(1)}$] Let $\tilde{\zeta}$ be a genus-$(g+1)$ SPWF obtained by applying flip and slip to $\zeta$. 
Then, there exist elements $\alpha, d, j, \tilde{c}_2,\ldots,\tilde{c}_n, \varphi$ satisfying the conditions $W_1^{\prime}, W_2^{\prime}, W_3^{\prime}$ and $W_4^{\prime}$ such that the sequence $(\tilde{c}_1,\ldots,\tilde{c}_n,\tilde{c}_1,\alpha^\prime,d,\tilde{\alpha})$ gives a Williams diagram $\tilde{\zeta}$, where $\tilde{c}_1=j^{-1}(c_1)$, $\tilde{\alpha}$ is the closure of $j^{-1}(\Int{\alpha})$ in $\Sigma_{g+1}$, and $\alpha^\prime$ is defined as follows:
\[
\alpha^\prime =\left(\varphi^{-1}\cdot t_{t_{\tilde{c}_1}(\tilde{c}_2)}\cdot\cdots\cdot t_{t_{\tilde{c}_n}(\tilde{c}_1})\right)(\tilde{\alpha}).  
\]

\item[$\mathrm{(2)}$] Let $\alpha, d, j, \tilde{c}_2,\ldots,\tilde{c}_n$ and $\varphi$ be elements satisfying the conditions $W_1^{\prime}, W_2^{\prime}, W_3^{\prime}$ and $W_4^{\prime}$. 
Suppose that $g$ is greater than or equal to $2$. 
We take simple closed curves $\alpha^\prime, \tilde{\alpha}$ as in $\mathrm{(1)}$. 
Then, there exists a genus-$(g+1)$ SPWF $\tilde{\zeta}$ obtained by applying flip and slip to $\zeta$ such that $(\tilde{c}_1,\ldots,\tilde{c}_n, \tilde{c}_1,\alpha^\prime,d, \tilde{\alpha})$ is a Williams diagram of $\tilde{\zeta}$. 

\end{enumerate}

\end{thm}

\begin{exmp}\label{exmp_trivialfibration}

Let $p_1: S^2\times \Sigma_k\rightarrow S^2$ be the projection onto the first component ($k\geq 0$). 
By applying a birth (for details about this move, see \cite{Lek} or \cite{Wil}, for example), we can change $p_1$ into a genus-$(k+1)$ SPWF with two cusps. 
We then apply a flip and slip move to this SPWF $m$ times. 
As a result, we obtain a genus-$(k+m+1)$ SPWF on the manifold $S^2\times \Sigma_k$. 
We denote this fibration by $\tilde{p}_1^{(m)}: S^2\times \Sigma_k\rightarrow S^2$. 

\vspace{.8em}
\noindent
{\bf Claim.} A Williams diagram of $\tilde{p}_1^{(m)}$ corresponds to $(d_0,d_1,\ldots,d_{2m},d_{2m+1},d_{2m},\ldots,d_1)$, where $d_i\subset \Sigma_{k+m+1}$ is a simple closed curve described in the left side of Figure \ref{scctrivialfibration}. 

\begin{figure}[htbp]
\begin{center}
\includegraphics[width=150mm]{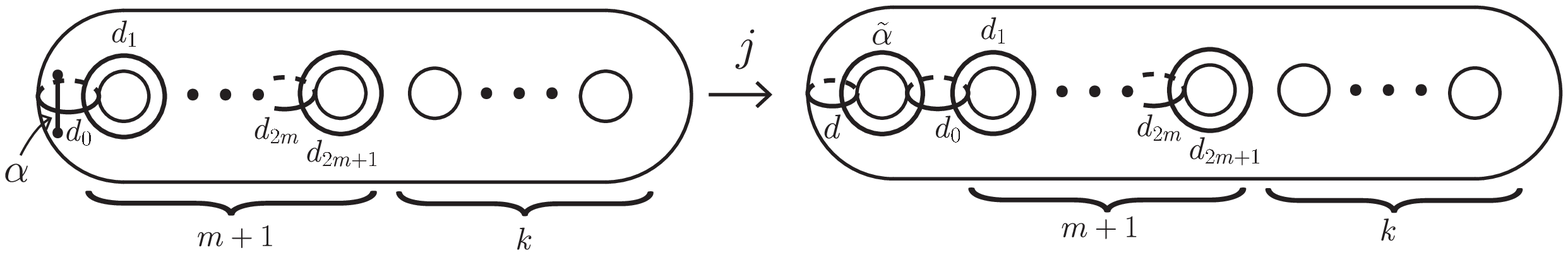}
\end{center}
\caption{simple closed curves in the genus-$(k+m+1)$ closed surface $\Sigma_{k+m+1}$. }
\label{scctrivialfibration}
\end{figure}

We prove this claim by induction on $m$. 
The claim is obvious when $m=0$. 
We assume that $m>0$. 
For simplicity, we denote the Dehn twist along the curve $d_i$ by $i$ and its inverse by $\bar{i}$. 
For an integer $n>0$, let $S_{n}$ be a regular neighborhood of the union $d_0 \cup \cdots\cup d_n$. 
By direct calculation, we can prove the following relation in $\MCG{(\overline{S_{n}}, \partial \overline{S_{n}})}$: 
\begin{equation}\label{relation_trivialfibration}
t_{t_{d_0}(d_{1})}\cdot \cdots \cdot t_{t_{d_{n-1}}(d_{n})}\cdot t_{t_{d_n}(d_{n-1})}\cdot\cdots \cdot t_{t_{d_{1}}(d_{0})} =\begin{cases}
\bar{0}^{4}\cdot (01)^{3} & (n=1),  \\
\bar{0}^{2n+2}\cdot (01\cdot\cdots\cdot n)^{n+2}\cdot (\bar{2}\bar{3}\cdot\cdots\cdot \bar{n})^{n} & (n\geq 2). 
\end{cases}
\end{equation}
By induction hypothesis, a sequence $(d_0,\ldots, d_{2m-2}, d_{2m-1}, d_{2m-2},\ldots, d_1)$ is a Williams diagram of $\tilde{p}_1^{(m-1)}$. 
We will stabilize this diagram by using Theorem \ref{mainalgorithm_SPWF}. 
We take a path $\alpha\subset \Sigma_{k+(m-1)+1}$ as in the left side of Figure \ref{scctrivialfibration}. 
Let $j:\Sigma_{k+m}\setminus \partial \alpha \rightarrow \Sigma_{k+m+1}\setminus d$ be a diffeomorphism, where $d$ is a non-separating simple closed curve.
By using $j$, we regard $d_i$ as a curve in $\Sigma_{k+m+1}$. 
It is easy to see that the element $t_{t_{d_0}(d_{1})}\cdot \cdots \cdot t_{t_{d_{2m-2}}(d_{2m-1})}\cdot t_{t_{d_{2m-1}}(d_{2m-2})}\cdot\cdots \cdot t_{t_{d_{1}}(d_{0})}$ is contained in the group $\MCG{(\Sigma_{k+m+1})}(d,d_0)$. 
Moreover, by the relation (\ref{relation_trivialfibration}), we can calculate the image under $\Phi_{d_0}: \MCG{(\Sigma_{k+m+1})}(d,d_0)\rightarrow \MCG{(\Sigma_{k+m})}(d)$ as follows: 
{\allowdisplaybreaks
\begin{align*}
& \Phi_{d_0}(t_{t_{d_0}(d_{1})}\cdot \cdots \cdot t_{t_{d_{2m-2}}(d_{2m-1})}\cdot t_{t_{d_{2m-1}}(d_{2m-2})}\cdot\cdots \cdot t_{t_{d_{1}}(d_{0})}) \\
= & \begin{cases}
\Phi_{d_0}(\bar{0}^{4}\cdot (01)^{3}) & (m=1),  \\
\Phi_{d_0}(\bar{0}^{4m+2}\cdot (01\cdot\cdots\cdot 2m-1)^{2m+1}\cdot (\bar{2}\bar{3}\cdot\cdots\cdot \overline{2m-1})^{2m-1}) & (m\geq 2). 
\end{cases} \\
= & \text{id}, 
\end{align*}
}
where the last equality is proved by the chain relation of the mapping class group. 
Note that this equality still holds in the group $\MCG{(\overline{S}, \partial \overline{S})}$, where $S$ is a regular neighborhood of the union $d\cup d_0\cup\cdots\cup d_{2m-1}$. 
We put $\varphi=t_{t_{d_0}(d_{1})}\cdot \cdots \cdot t_{t_{d_{2m-2}}(d_{2m-1})}\cdot t_{t_{d_{2m-1}}(d_{2m-2})}\cdot\cdots \cdot t_{t_{d_{1}}(d_{0})}\in\MCG{(\Sigma_{k+m+1})}(d,d_0)$. 
The elements $\alpha,d,j,d_0,\ldots,d_{2m-1},\varphi$ satisfy the conditions $C_1^W$, $C_2^W$, $C_3^W$ and $C_4^W$. 
Thus, by Theorem \ref{mainalgorithm_SPWF}, $(d_0,\ldots, d_{2m-2}, d_{2m-1}, d_{2m-2},\ldots, d_1, d_0, \tilde{\alpha}, d, \tilde{\alpha})$ is a Williams diagram of $\tilde{p}_1^{(m)}$. 
Note that this still holds when the genus of $\tilde{p}_1^{(m-1)}$ is less than $3$ since the above calculation of elements of mapping class groups can be done in regular neighborhoods of curves. 
This proves the claim on Williams diagrams of $S^2\times \Sigma_k$.

\end{exmp}

\begin{rem}\label{rem_stepfibration}

It is known that there exists a genus-$k$ SPWF $q: S^2\times \Sigma_{k-1}\# S^1\times S^3 \rightarrow S^2$ without cusp singularities for $k\geq 1$. 
This was introduced in \cite{Ba}, and was called the {\it step fibration}. 
By the same argument as in Example \ref{exmp_trivialfibration}, we can prove that $(d_0,d_1,\ldots d_{2m-1},d_{2m},d_{2m-1},\ldots,d_1)$ is a Williams diagram of the fibration obtained by applying flip and slip to $q$ $m$ times. 

We can also prove the claims on Williams diagrams of $S^2\times \Sigma_k$ and $S^2\times \Sigma_{k-1}\# S^1\times S^3$ by using Lemma \ref{lem_surgeryformula}. 

\end{rem}

\begin{exmp}\label{exmp_2S1S3}

We next construct a Williams diagram of $\# 2 S^1\times S^3$, which will be used to construct a Williams diagram of $S^4$. 
To do this, we first prove the following lemma. 

\begin{lem}\label{genus2fibration2S1S3}

$\#2 S^1\times S^3$ admits a genus-$2$ SPWF $\zeta$ without cusps. 
Moreover, an attaching loop $\beta_x$ of this fibration is described as in the left side of Figure \ref{sccgenus2fibration2S1S3}, where $e_0$ is a vanishing cycle of indefinite fold singularity. 

\begin{figure}[htbp]
\begin{center}
\includegraphics[width=105mm]{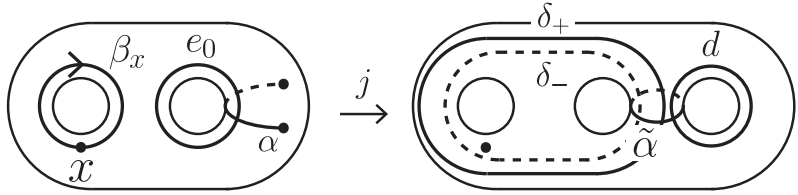}
\end{center}
\caption{$\tilde{\alpha}$ is the closure of $j(\Int{\alpha})$ in $\Sigma_3$. }
\label{sccgenus2fibration2S1S3}
\end{figure}

\end{lem}

\begin{proof}[Proof of Lemma \ref{genus2fibration2S1S3}]

It is easy to show that there exists a genus-$2$ SPWF $\zeta$ without cusps and whose attaching loop is $\beta_x$ which is described in Figure \ref{sccgenus2fibration2S1S3}. 
Furthermore, we can draw a Kirby diagram of the total space of $\zeta$ as described in Figure \ref{Kirbygenus2fibration2S1S3}. 
It can be easily shown by Kirby calculus that this manifold is diffeomorphic to $\#2 S^1\times S^3$. 

\begin{figure}[htbp]
\begin{center}
\includegraphics[width=100mm]{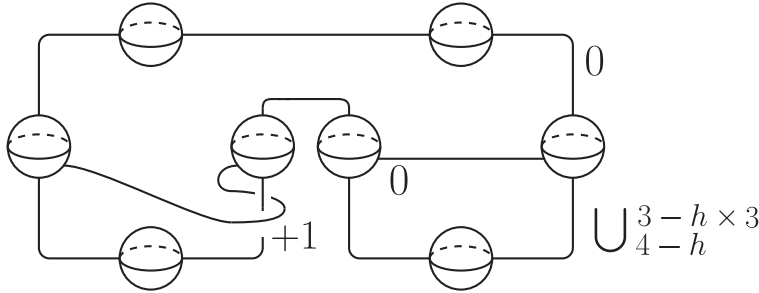}
\end{center}
\caption{Kirby diagram of the fibration $\zeta$. }
\label{Kirbygenus2fibration2S1S3}
\end{figure}

\end{proof}

We take a path $\alpha\subset\Sigma_2$ as in the left side of Figure \ref{sccgenus2fibration2S1S3}. 
We also take a diffeomorphism $j: \Sigma_2\setminus \partial \alpha\rightarrow \Sigma_{3}\setminus d$, where $d$ is a non-separating simple closed curve in $\Sigma_3$, so that the closure of $j(\Int{\alpha})$ is a simple closed curve. 
Let $\delta_+, \delta_{-}\subset \Sigma_3$ be simple closed curves descried as in the right side of Figure \ref{sccgenus2fibration2S1S3}. 
We define an element $\varphi\in \MCG{(\Sigma_3, x)}(d,e_0)$ as follows: 
\[
\varphi= Push(\beta_x)\cdot t_{\delta_+}\cdot t_{\delta_{-}}^{-1}. 
\]
It is easy to see that this element satisfies $\Phi_{d}^{x}(\varphi)=Push(\beta_x)$ and $\Phi_{e_0}^{x}(\varphi)=\text{id}$. 
Thus, the elements $\alpha,d,j,e_0,\varphi$ satisfy the conditions $W_1^{\prime}$, $W_2^{\prime}$, $W_3^{\prime}$ and $W_4^{\prime}$. 
By Theorem \ref{mainalgorithm_SPWFwithsection}, $(e_0, \alpha^\prime,d,\tilde{\alpha})$ is a Williams diagram of the fibration obtained by applying flip and slip to $\zeta$, where $\alpha^\prime= (\varphi^{-1})(\tilde{\alpha})$ (see Figure \ref{sccgenus3fibration2S1S3}). 

\begin{figure}[htbp]
\begin{center}
\includegraphics[width=120mm]{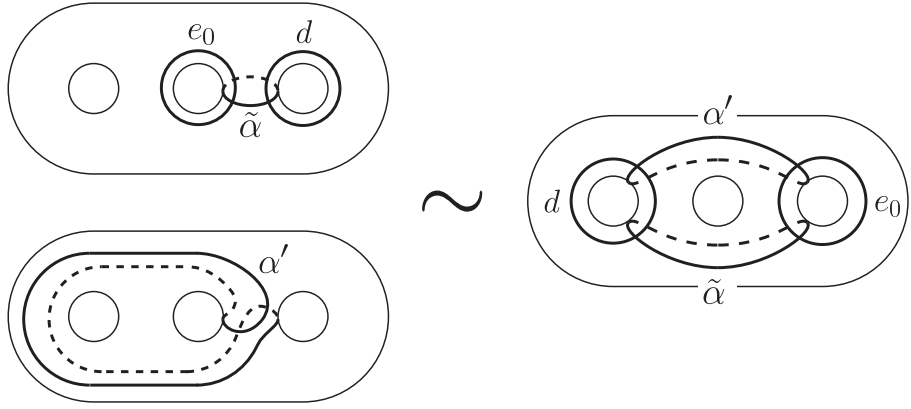}
\end{center}
\caption{simple closed curves contained in a Williams diagram of $\#2 S^1\times S^3$.}
\label{sccgenus3fibration2S1S3}
\end{figure}

\begin{rem}

More generally, we can obtain a genus-$(m+2)$ Williams diagram of $\#2S^1\times S^3$ by looking at vanishing cycles of a fibration obtained by applying flip and slip to $\zeta$ $m$ times. 

\vspace{.8em}

\noindent
{\bf Claim.} Let $e_1,\ldots , e_{3m+1}$ be simple closed curves in $\Sigma_{m+2}$ described in Figure \ref{scchighgenusfibration2S1S3}. 
The following sequence is the Williams diagram of $\#2S^1\times S^3$:  
\[
(e_1,e_2,\ldots, e_{2m-1},e_{2m},e_{2m+1}, e_{2m+2}, e_{2m-1}, e_{2m+3}, e_{2m-3},\ldots, e_{3m}, e_3,e_{3m+1}).
\]
\begin{figure}[htbp]
\begin{center}
\includegraphics[width=150mm]{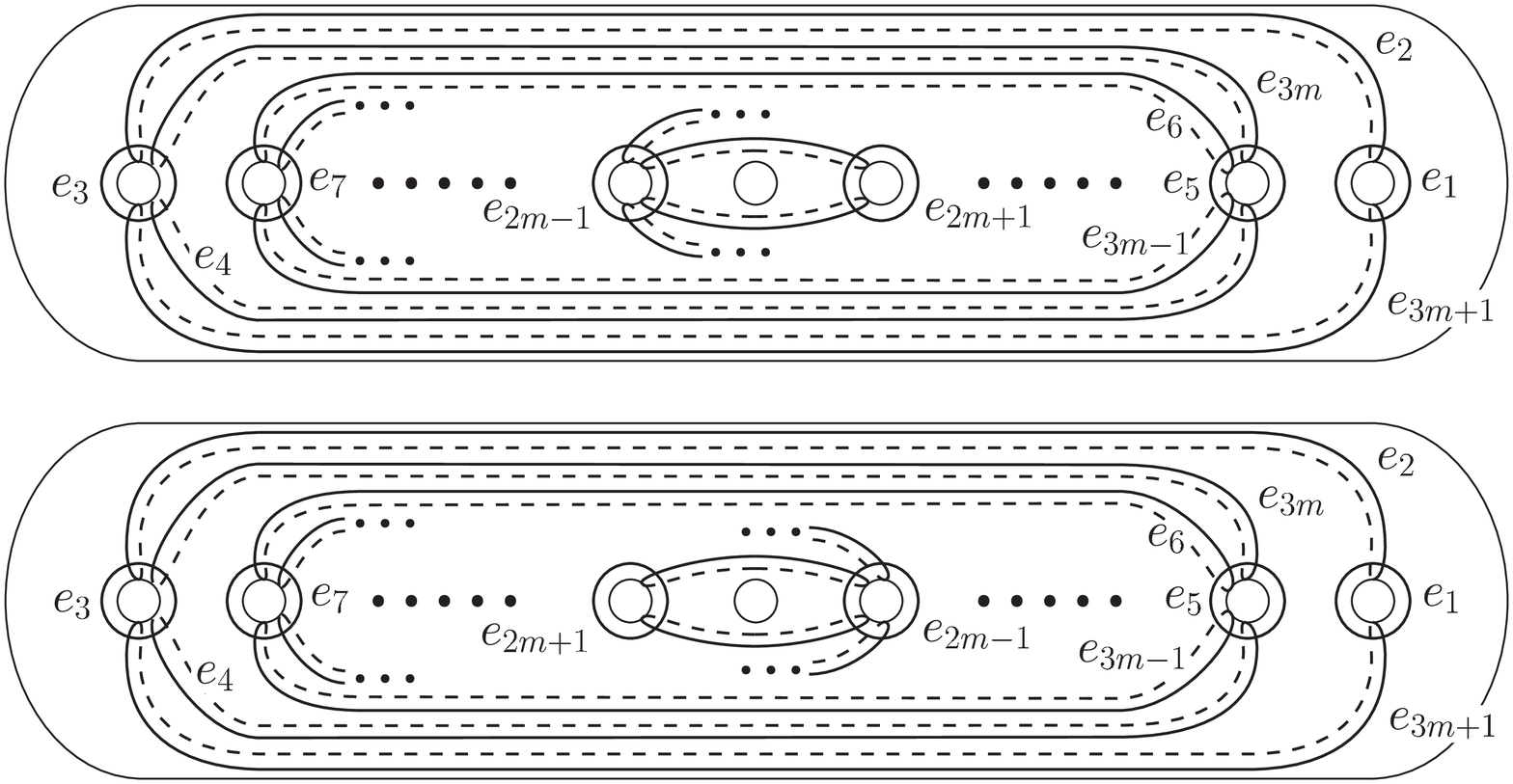}
\end{center}
\caption{the upper figure describes simple closed curves $e_1,\ldots, e_{3m+1}$ in the case $m$ is even, while the lower figure describes simple closed curves $e_1,\ldots, e_{3m+1}$ in the case $m$ is odd. }
\label{scchighgenusfibration2S1S3}
\end{figure}

\end{rem}

\end{exmp}

Before looking at the next example, we prove the following lemma. 

\begin{lem}\label{lem_surgeryformula}

Let $(c_1,\ldots, c_n)$ be a genus-$g$ Williams diagram of an SPWF $\zeta:M\rightarrow S^2$. 
We take a simple closed curve $\gamma \subset \Sigma_g$ which intersects $c_{i_0}$ at a unique point transversely. 
Then there exists a genus-$g$ SPWF $\zeta_s: M_s\rightarrow S^2$ whose Williams diagram is $(c_1,\ldots,c_{i_0-1}, c_{i_0},\gamma, c_{i_0}, c_{i_0+1},\ldots , c_n)$. 
Moreover, if $g$ is greater than or equal to $3$, the manifold $M_s$ is obtained from $M$ by applying surgery along $\gamma$, where we regard $\gamma$ as in a regular fiber of $\zeta$. 

\end{lem}

\begin{proof}[Proof of Lemma \ref{lem_surgeryformula}]

By applying cyclic permutation to the sequence $(c_1,\ldots,c_n)$ if necessary, we can assume that $i_0=1$. 
It is easy to see that the element $t_{t_{c_1}(\gamma)}\cdot t_{t_{\gamma}(c_1)}$ is contained in the kernel of $\Phi_{c_1}$. 
Thus, the product $t_{t_{c_1}(\gamma)}\cdot t_{t_{\gamma}(c_1)}\cdot t_{t_{c_1}(c_2)}\cdot \cdots \cdot t_{t_{c_n}(c_1)}$ is also contained in the kernel of $\Phi_{c_1}$. 
This implies existence of a genus-$g$ simplified broken Lefschetz fibration with vanishing cycles $(c_1,t_{c_1}(\gamma),t_{\gamma}(c_1),t_{c_1}(c_2),\ldots,t_{c_n}(c_1))$. 
Such a fibration can be changed into an genus-$g$ SPWF $\zeta_s:M_s\rightarrow S^2$ with Williams diagram $(c_1,\gamma,c_1,\ldots ,c_n)$ by applying sink. 
To prove the statement on $M_s$, we look at the submanifold $S$ of $M$ satisfying the following conditions:
\begin{enumerate}

\item the image $f(S)$ is a disk and the intersection $f(S\cap \mathcal{S}_f)$ forms a connected arc without cusps, 

\item a vanishing cycle of indefinite folds in $f(S)$ is $c_1$, 

\item the restriction $f|_{S\setminus f^{-1}(\mathcal{S}_f)}: S\setminus \mathcal{S}_f\rightarrow f(S)\setminus f(\mathcal{S}_f)$ is a disjoint union of trivial fibration,  

\item the higher genus fiber of $f|_{S}: S\rightarrow f(S)$ is a regular neighborhood of the union $c_1\cup \gamma$. 

\end{enumerate}

\noindent
We can easily draw a Kirby diagram of $S$ as in the left side of Figure \ref{Kirbysurgeredsubmanifold}. 
This diagram implies that $S$ is diffeomorphic to $S^1\times D^3$, and that a generator of $\pi_1(S)$ corresponds to a simple closed curve $\gamma$. 
Let $\overline{S}$ be a manifold which is described in the right side of Figure \ref{Kirbysurgeredsubmanifold}. 
This manifold admit a fibration to $D^2$ with connected indefinite fold, which forms an arc, and two Lefschetz singularities. 
Furthermore, a regular fiber of the fibration is either a genus-$1$ surface with one boundary component or a disk. 
By Kirby calculus, we can prove that this manifold is diffeomorphic to $D^2\times S^2$. 
By the construction of the fibration $\zeta_s$, the manifold $M_s$ can be obtained by removing $S$ from $M$, and then attaching $\overline{S}$ along the boundary. 
This completes the proof of Lemma \ref{lem_surgeryformula}. 

\begin{figure}[htbp]
\begin{center}
\includegraphics[width=80mm]{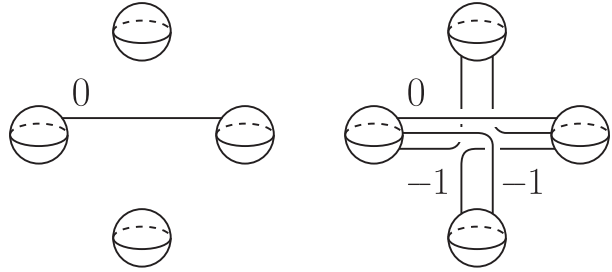}
\end{center}
\caption{Left: a Kirby diagram of $S$. Right: a Kirby diagram of $\overline{S}$}
\label{Kirbysurgeredsubmanifold}
\end{figure}

\end{proof}

\begin{exmp}\label{exmp_S4etc}

Let $e_1,e_2,e_3,e_4$ be simple closed curves in $\Sigma_3$ as described in Figure \ref{sccgenus3fibrations}. 
As is shown, a sequence $(e_1,e_2,e_3,e_4)$ is a Williams diagram of $\#2S^1\times S^3$. 
We take a curve $\gamma_i$ ($i=1,2,3,4$) as shown in Figure \ref{sccgenus3fibrations}. 
The curve $\gamma_1$ intersects $e_1$ at a unique point transversely. 
By Lemma \ref{lem_surgeryformula}, a sequence $(e_1,\gamma_1,e_1,e_2,e_3,e_4)$ is a Williams diagram of some $4$-manifold obtained by applying surgery to $\#2S^1\times S^3$. 
Indeed, we can prove by Kirby calculus that this diagram represents the manifold $S^1\times S^3$. 
In the same way, we can prove the following correspondence between Williams diagrams and $4$-manifolds: 

\begin{center}
{\renewcommand\arraystretch{1.2}
\begin{tabular}{|c|c|}
\hline
Williams diagram & corresponding $4$-manifold \\
\hline
$(e_1,\gamma_1,e_1,e_2,e_3,\gamma_4,e_3,e_4)$ & $S^4$ \\
\hline
$(e_1,\gamma_1,e_1,e_2,e_3,\gamma_3,e_3,e_4)$ & $S^1\times S^3\#S^2\times S^2$ \\
\hline
$(e_1,\gamma_1,e_1,e_2,e_3,\gamma_4,e_3,\gamma_4,e_3,e_4)$ & $S^2\times S^2$ \\
\hline
$(e_1,\gamma_1,\gamma_2,\gamma_1,e_1,e_2,e_3,\gamma_4,e_3,e_4)$ & $\mathbb{CP}^2\#\overline{\mathbb{CP}^2}$ \\
\hline
\end{tabular}
}
\end{center}

\vspace{.3em}
\noindent
In particular, we have obtained two genus-$3$ SPWFs on $S^2\times S^2$ which is derived from the following two Williams diagrams: the diagram $(d_0, d_1, d_2, d_3, d_4, d_5, d_4, d_3, d_2, d_1)$ in Example \ref{exmp_trivialfibration}, and the diagram $(e_1,\gamma_1,e_1,e_2,e_3,\gamma_4,e_3,\gamma_4,e_3,e_4)$ as above. 
The SPWF corresponding to the former diagram is homotopic to the projection $p_1:S^2\times S^2\rightarrow S^2$ onto the first projection. 
Indeed, this SPWF was constructed by applying birth and flip and slip to $p_1$. 
On the other hand, it is easy to prove (by Kirby calculus, for example) that a regular fiber of the SPWF corresponding to the latter diagram is null-homologous in $S^2\times S^2$. 
Thus, two genus-$3$ SPWFs above are not homotopic. 
In the same way, we can prove that two SPWFs on $S^1\times S^3\# S^2\times S^2$ derived from the following two diagrams are not homotopic: the diagram $(d_0,d_1,d_2,d_3,d_4,d_3,d_2,d_1)$ which is obtained by applying flip and slip to the step fibration twice (see Remark \ref{rem_stepfibration}), and the diagram $(e_1,\gamma_1,e_1,e_2,e_3,\gamma_3,e_3,e_4)$ as above. 

\begin{figure}[htbp]
\begin{center}
\includegraphics[width=60mm]{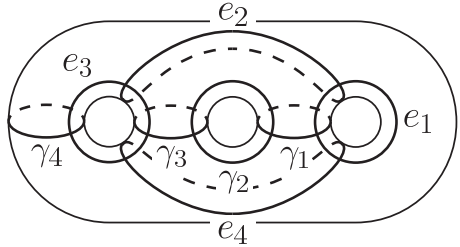}
\end{center}
\caption{simple closed curves in $\Sigma_3$. }
\label{sccgenus3fibrations}
\end{figure}
\end{exmp}

\end{document}